\documentclass[11pt, final]{amsart}
\usepackage{amsmath, amsxtra, amssymb, mathdots}
\usepackage{verbatim}
\usepackage[margin=1.2in]{geometry}
\usepackage{multirow, tikz}
\usepackage[notref,notcite]{showkeys}
\usepackage[colorlinks, linkcolor=blue, citecolor=black]{hyperref}

\numberwithin{equation}{section}

\theoremstyle{plain}
\newtheorem{theorem}{Theorem}[section]
\newtheorem{prop}[theorem]{Proposition}
\newtheorem{corollary}[theorem]{Corollary}

\newtheorem{lemma}[theorem]{Lemma}

\theoremstyle{definition}
\newtheorem{definition}[theorem]{Definition}
\newtheorem{example}[theorem]{Example}
\newtheorem{remark}[theorem]{Remark}
\newtheorem{conjecture}[theorem]{Conjecture}

\newcommand{\B}{\mathcal{B}}
\newcommand{\CC}{\mathbb{C}}
\newcommand{\I}{\mathcal{I}}
\newcommand{\GG}{\mathbb{G}}
\newcommand{\M}{\overline{M}}
\newcommand{\HH}{\mathrm{H}}
\newcommand{\PP}{\mathbb{P}}
\renewcommand{\O}{\mathcal{O}}
\newcommand{\ZZ}{\mathbb{Z}}
\newcommand{\Aut}{\operatorname{Aut}}
\newcommand{\Proj}{\operatorname{Proj}}
\newcommand{\SL}{\operatorname{SL}}
\newcommand{\gitq}{/\hspace{-0.25pc}/}
\newcommand{\ra}{\rightarrow}
\newcommand{\co}{\colon\thinspace} 
\newcommand{\Spec}{\operatorname{Spec}}
\newcommand{\gm}{\mathbb{G}_m}
\newcommand{\Sym}{\operatorname{Sym}}

\newcommand*{\longhookrightarrow}
{\ensuremath{\lhook\joinrel\relbar\joinrel\relbar\joinrel\relbar\joinrel\relbar\joinrel\relbar\joinrel\relbar\joinrel\relbar\joinrel\rightarrow}}
\newcommand{\lspan}{\operatorname{span}}
\newcommand{\rank}{\operatorname{rank}}
\newcommand\Mg[1]{\overline{\mathcal{M}}_{#1}}

\begin{document}
\title{Finite Hilbert stability of (bi)canonical curves}
\author[Alper]{Jarod Alper}
\author[Fedorchuk]{Maksym Fedorchuk}
\author[Smyth]{David Ishii Smyth*}

\address[Alper]{Departamento de Matem\'aticas\\
Universidad de los Andes\\
Cra 1 No. 18A-10\\
Edificio H\\
Bogot\'a, 111711, Colombia} \email{jarod@uniandes.edu.co}

\address[Fedorchuk]{Department of Mathematics\\
Columbia University\\
2990 Broadway\\
New York, NY 10027}
\email{mfedorch@math.columbia.edu}

\address[Smyth]{Department of Mathematics\\
Harvard University\\
1 Oxford Street\\
Cambridge, MA 01238}
\email{dsmyth@math.harvard.edu}

\dedicatory{To Joe Harris on his sixtieth birthday}

\thanks{*The third author was partially supported by NSF grant 
DMS-0901095 during the preparation of this work.}

\begin{abstract} We prove that a generic canonically or bicanonically embedded smooth curve 
has semistable $m^{th}$ Hilbert points for all $m \geq 2$. 
We also prove that a generic bicanonically embedded smooth curve has stable $m^{th}$ Hilbert points for all $m\geq 3$. 
In the canonical case, this is accomplished by proving finite Hilbert semistability of special singular curves with $\GG_m$-action, 
namely the canonically embedded {\em balanced ribbon} and the canonically embedded 
{\em balanced double $A_{2k+1}$-curve}. In the bicanonical case, we prove finite Hilbert stability of special hyperelliptic curves, 
namely {\em Wiman curves}. 
Finally, we give examples of canonically embedded smooth curves 
whose $m^{th}$ Hilbert points are non-semistable for low values of $m$, 
but become semistable past a definite threshold.
\end{abstract}
\maketitle

\setcounter{tocdepth}{2}
\tableofcontents
\clearpage

\section{Introduction}
Geometric Invariant Theory (GIT) was developed by Mumford in order to construct quotients in algebraic geometry,
and in particular to construct moduli spaces. To use GIT to construct a moduli space one must typically prove that 
a certain class of embedded varieties has stable or semistable Hilbert points.
The prototypical example of a stability result is Gieseker and Mumford's asymptotic stability theorem 
for pluricanonically embedded curves \cite{mumford-stability, gieseker, gieseker-CIME}:  
\begin{theorem}[Asymptotic Stability] 
\label{T:asymptotic-stability}
Suppose $C \subset \PP \HH^0\bigl(C, K_{C}^n\bigr)$ is a smooth curve embedded by the complete linear system 
$|K_{C}^n|$, where $n \geq 1$. Then the $m^{th}$ Hilbert point of $C$ is stable for all $m \gg 0$.
\end{theorem}

Gieseker and Mumford's arguments are non-effective, and there is no known bound on how large $m$ 
must be in order to obtain the conclusion of the theorem. 
In light of this theorem, it is natural to ask: for which finite values of $m$ do pluricanonically 
embedded smooth curves have stable or semistable Hilbert points? This has been a basic open problem in GIT 
since the pioneering work of Gieseker and Mumford, but has gained renewed interest from recent work 
of Hassett and Hyeon on the log minimal model program for $\M_g$.
Indeed, Hassett and Hyeon observed that a stability result 
for finite Hilbert points of canonically and bicanonically embedded smooth curves 
would enable one to use GIT to construct a sequence of new projective birational models 
of $M_{g}$ that would constitute steps of the log minimal model program for $\M_{g}$ \cite{hassett-hyeon_flip}. 
In this paper, we prove the requisite stability result. 
\begin{theorem}[Main Result]\label{T:main} 
\hfill
\begin{enumerate}
\item
If $C$ is a generic canonically or bicanonically embedded smooth curve, 
then the $m^{th}$ Hilbert point of $C$ is semistable for every $m \geq 2$. 
\item
If $C$ is a generic bicanonically embedded smooth 
curve, then the $m^{th}$ Hilbert point of $C$ is stable for every $m \geq 3$. 
\end{enumerate}
\end{theorem}
Part (1) of the main result is proved in Corollaries \ref{C:ribbon} (odd genus canonical), 
\ref{C:main-trigonal} (even genus canonical), and Theorem \ref{T:bicanonical} (bicanonical case). 
Part (2) of the main result is proved in Theorem \ref{T:bicanonical}. 
This is, to our knowledge, the first example of a result in which the (semi)stability of \emph{all} 
Hilbert points of a given variety is established by a uniform method. In the case of canonically and bicanonically
embedded curves, we recover a weak form of the asymptotic stability theorem by a much simpler proof. 
Furthermore, as a sidelight to our main result, we give an example of an embedded {\em smooth} curve 
whose $m^{th}$ Hilbert point changes from semistable to non-semistable as $m$ decreases (Theorem \ref{T:bielliptic-change}).
We will explain our method of proof in the next section. First, however, let us 
conclude this introduction by describing a fascinating application of the main result, 
anticipated in the work of Hassett and Hyeon \cite{hassett-hyeon_flip}, 
and by considering prospects for future generalizations.

Fix $g \geq 2$, $n \geq 1$, $m\geq 2$, and set $r=(2n-1)(g-1)-1$ if $n\geq 2$, and $r=g-1$ if $n=1$. 
To an $n$-canonically embedded smooth genus $g$ curve $C$ we associate its $m^{th}$ Hilbert point 
$[C]_m \in \PP W_m$; these are defined in more detail in Section \ref{S:kempf} below. 
We denote by $\overline{H}_{g,n}^{\, m}$ the closure in $\PP W_m$ of the locus 
of $m^{th}$ Hilbert points of $n$-canonically embedded smooth curves of genus $g$. Then the $\SL(r+1)$-action 
on $\overline{H}_{g,n}^{\, m}$ admits a natural linearization $\O(1)$, which defines an open locus 
$(\overline{H}_{g,n}^{\, m})^{ss} \subset \overline{H}_{g,n}^{\, m}$ of semistable points. 
{\em Assuming that $(\overline{H}_{g,n}^{\, m})^{ss}$ is non-empty}, one obtains a GIT quotient
\begin{align*}
(\overline{H}_{g,n}^{\, m})^{ss} \gitq \SL(r+1):=\Proj \bigoplus_{k\geq 0} \HH^0\bigl(\, \overline{H}_{g,n}^{\, m}, \O(k)\bigr)^{\SL(r+1)}
\end{align*} 
as a projective variety associated to the algebra of $\SL(r+1)$-invariant functions 
in the homogenous coordinate ring of $\overline{H}_{g,n}^{\, m}$.
 
When $m \gg 0$, the critical assumption $\bigl(\overline{H}_{g,n}^{\, m} \bigr)^{ss} \neq \varnothing$ is satisfied 
by Theorem \ref{T:asymptotic-stability}, and the corresponding quotients have been analyzed using 
GIT \cite{gieseker,gieseker-CIME,schubert,hassett-hyeon_contraction,hassett-hyeon_flip,hyeon-lee_genus3,hyeon-morrison}. 
The results of this analysis can be summarized as follows:
\begin{align*}
\bigl(\overline{H}_{g,n}^{\, m} \bigr)^{ss} \gitq \SL(r+1) \simeq &
\begin{cases}
\M_{g} &\text{ if $n \geq 5, m \gg 0$},\\
 \M_{g}^{\, ps} &\text{ if $n = 3,4, m \gg 0$},\\
\M_{g}^{\, hs} &\text{ if $n = 2, m \gg 0$}.
\end{cases}
\end{align*}
Here, $\M_{g}^{\, ps}$ is the moduli space of pseudostable curves, in which elliptic tails have been replaced by cusps, and $\M_{g}^{\, hs}$ is 
the moduli space of $h$-semistable curves, in which elliptic bridges have been replaced by tacnodes. Furthermore, the birational 
transformations $\M_{g} \ra \M_{g}^{\, ps} \dashrightarrow \M_{g}^{\, hs}$ constitute the first two steps of the log minimal model program, 
namely the first divisorial contraction and the first flip \cite{hassett-hyeon_contraction,hassett-hyeon_flip}.

The key point is that the next stage of the log minimal model program cannot be constructed using an asymptotic stability result. 
Indeed, an examination of the 
formula for the divisor class of the polarization on the GIT quotient 
$\bigl(\overline{H}_{g,n}^{\, m}\bigr)^{ss} \gitq \SL(r+1)$ suggests that the next model occurring in the log minimal model program 
should be 
$\bigl(\overline{H}_{g,2}^{\, 6}\bigr)^{ss} \gitq \SL(3g-3)$.  
Thus, in marked contrast to the cases $n \geq 3$, where finite Hilbert linearizations are not expected 
to yield new birational models of $\M_{g}$, it is widely anticipated that in the cases 
$n=1, 2$, there will exist several values of $m$ at which the corresponding GIT 
quotients undergo nontrivial birational modifications caused by the fact that curves with worse than 
nodal singularities become semistable for low values of $m$. 
For $n=1$ we expect the number 
of threshold values of $m$ at which $\bigl(\overline{H}_{g,n}^{\, m}\bigr)^{ss}$ changes 
to grow with $g$, while for $n=2$ the only interesting values are $m\leq 6$, irrespectively of $g$;
for a detailed analysis of the expected threshold values of $m$ see \cite{handbook} and 
\cite{afs_preprint}. Until now, the main obstacle to verifying these expectations has been proving 
$\bigl(\overline{H}_{g,n}^{\, m}\bigr)^{ss} \neq \varnothing$ for explicit, finite values of $m$ and arbitrary genus $g$. 
Theorem \ref{T:main} removes this obstacle, and thus opens the door to analyzing a 
whole menagerie of new GIT quotients $\bigl(\overline{H}_{g,n}^{\, m}\bigr)^{ss} \gitq \SL(r+1)$.

Finally, let us discuss a slight sharpening of our main result which follows naturally from the methods employed in this 
paper. We observe that the canonically embedded curve of even genus 
for which we establish finite Hilbert semistability in Section \ref{S:monomial-bases-A-curve} 
is in fact trigonal, i.e. it lies in the closure of the locus of canonically embedded 
smooth trigonal curves. Similarly, in Section \ref{S:monomial-bases-rosary}, 
we prove the finite Hilbert semistability of the bicanonically embedded curve of odd genus, which is easily seen to be 
in the closure of the locus of bicanonically embedded smooth bielliptic curves.
From these observations, we obtain the following result:
\begin{theorem}[Stability of trigonal and bielliptic curves]\label{T:main-trigonal} 
\hfill
\begin{enumerate}
\item Suppose $C \subset \PP \HH^0\bigl(C, K_{C}\bigr)$ is a generic canonically embedded smooth trigonal curve of even genus. 
Then the $m^{th}$ Hilbert point of $C$ is semistable for every $m \geq 2$. \par
\item Suppose $C \subset \PP \HH^0\bigl(C, K_{C}^2\bigr)$ is a generic bicanonically embedded smooth bielliptic curve of odd genus. 
Then the $m^{th}$ Hilbert point of $C$ is semistable for every $m \geq 2$.
\end{enumerate}
\end{theorem}
This result naturally raises the questions: 
Is it true that {\em all} canonically embedded smooth trigonal curves have semistable $m^{th}$ Hilbert points for $m \geq 2$? 
Similarly, do other curves with low Clifford index, such as canonical bielliptic curves, have this property? 
Surprisingly, the answer to both questions is no. In Section \ref{S:non-semistability} of this paper, we prove that the $m^{th}$ Hilbert 
point of a canonically embedded smooth bielliptic curve is non-semistable below 
a certain definite threshold value of $m$ (depending on $g$), 
while the $m^{th}$ Hilbert point of a generic canonically embedded bielliptic
curve of odd genus is semistable for large values of $m$. As for trigonal curves, it is not difficult to see that the 
$2^{nd}$ Hilbert point of a canonically embedded 
trigonal curve with positive Maroni invariant is non-semistable; see \cite[Corollary 3.2]{fedorchuk-jensen}. 
On the other hand, in Section \ref{S:non-semistability} we give heuristic reasons for believing that a 
canonically embedded smooth trigonal curve 
should have semistable $m^{th}$ Hilbert points for $m\geq 3$. 

\subsubsection*{Notation and conventions} We work over the field of complex numbers $\CC$. 
In particular, we denote $\gm:=\Spec \CC[t,t^{-1}]$. 
In Section \ref{S:wiman}, we use the term {\em multiset} to denote a collection of elements with possibly repeating elements.

\subsection*{Acknowledgements}
We learned about the problem of GIT stability of finite Hilbert points many years ago from Brendan Hassett's 
talks on the log minimal model program for $\M_{g}$. 
Over the past several years we learned about many aspects of GIT
from conversations with Ian Morrison and David Hyeon, 
as well as through their many papers on the topic. In addition, we gained a great deal 
from conversations with Aise Johan de Jong, Anand Deopurkar, David Jensen, and David Swinarski.

\section{GIT background}\label{S:kempf}
The proof of our main result is surprisingly simple. In the canonical (resp., bicanonical) case, 
we exhibit a curve $C$ such that the action of $\Aut(C)$ on $V=\HH^0\bigl(C, \omega_C\bigr)$ 
(resp., $V=\HH^0 \bigl(C, \omega_C^2\bigr)$) is multiplicity-free, i.e. no representation occurs more than 
once in the decomposition of $V$ into irreducible $\Aut(C)$-representations. 
As Ian Morrison observed some thirty years ago, under
this hypothesis, powerful results of Kempf imply that the $m^{th}$ Hilbert point of $C$ 
is semistable if and only if it is semistable with respect to one-parameter subgroups of $\SL(V)$ 
which act diagonally on a {\em fixed basis} of $V$. Verifying stability with respect to the resulting fixed torus of $\SL(V)$
is a discrete combinatorial problem which we solve explicitly for every $m \geq 2$. 
We thus prove the semistability of all Hilbert points of $C$ and deduce the semistability of a generic smooth 
curve by openness of the semistable locus. In Section \ref{S:curves}
we will give a precise description of the (rather exotic) curves $C$ appearing in 
our argument. In this section, we recall the relevant definitions from GIT and explain the general framework 
for proving semistability of Hilbert points due to Mumford, as well as the aforementioned refinements of Kempf.

Let us begin by recalling the definition of the $m^{th}$ Hilbert point of an embedded scheme. If $X\subset \PP V$ is a closed subscheme 
such that the restriction map $\HH^0\bigl(\PP V, \O(m)\bigr) \ra \HH^0\bigl(X, \O_{X}(m)\bigr)$ is surjective (equivalently,
$h^1\bigl(X,\I_X(m)\bigr)=0$), set 
\[
W_m:= \bigwedge^{h^0\bigl(X, \O_{X}(m)\bigr)}\HH^0\bigl(\PP V, \O(m)\bigr)^{\vee}.
\]
The \emph{$m^{th}$ Hilbert point of $X \subset \PP V$} is a point $[X]_{m} \in \PP W_m$, defined as follows. First, consider the surjection
\begin{align*}
\HH^0\bigl(\PP V, \O(m)\bigr) \ra \HH^0\bigl(X,\O_X(m)\bigr) \ra 0.
\end{align*}
Taking the $h^0\bigl(X,\O_X(m)\bigr)$-fold wedge product and dualizing, 
we obtain the $m^{th}$ Hilbert point:
\begin{align*}
[X]_m:=\left[\bigwedge^{h^0\bigl(X, \O_{X}(m)\bigr)}\HH^0\bigl(\PP V, \O(m)\bigr) \ra \bigwedge^{h^0\bigl(X, \O_{X}(m)\bigr)} 
\HH^0\bigl(X,\O_X(m)\bigr) \ra 0 \right]^{\vee} 
\in \PP(W_m).
\end{align*}
Recall that if $W$ is any linear representation of $\SL(V)$, a point  $x \in \PP(W)$ is \emph{semistable} if the origin of $W$ is not contained in 
the closure of the orbit of $\tilde x \in W$, where $\tilde{x}$ is any lift of $x$. 
Thus, to show that a Hilbert point $[X]_m \in \PP (W_m)$ is semistable, we must prove that $0 \in W_m$ is not in the closure of 
$\SL(V) \cdot \widetilde{[X]}_m$, where $\widetilde{[X]}_m$ is any lift of $[X]_{m}$. 
An obvious necessary condition is that for any \emph{one-parameter subgroup} $\rho\co \Spec \CC[t,t^{-1}] \rightarrow \SL(V)$, we have 
$\lim_{t \rightarrow 0} \rho(t) \cdot \widetilde{[X]}_m \neq 0$.  A foundational theorem of Mumford asserts that this 
necessary condition is sufficient.
\begin{prop}[Hilbert-Mumford Numerical Criterion] 
Let $X\subset \PP V$ be as above.
The Hilbert point $[X]_m$ is semistable if and only if $\lim_{t \to 0} \rho(t) \cdot \widetilde{[X]}_m \neq 0$ 
for every one-parameter subgroup $\rho\co \Spec \CC[t,t^{-1}] \rightarrow \SL(V)$.
\end{prop}

Given a one-parameter subgroup $\rho\co \Spec \CC[t, t^{-1}] \rightarrow \SL(V)$, we may reformulate the condition 
$\lim_{t \rightarrow 0} \rho(t) \cdot \widetilde{[X]}_m \neq 0$ as follows. 
First, we may choose a basis $\{x_i\}_{i=0}^{r}$ of $V$ which diagonalizes the action of $\rho$. 
Then $\rho(t) \cdot x_i = t^{\rho_i}x_i$ 
for some integers $\rho_i$ satisfying $\sum_{i=0}^{r}\rho_i=0$. We call $\{x_i\}_{i=0}^r$ a {\em $\rho$-weighted basis}.
If we set $N_m :=h^0\bigl(X, \O_{X}(m)\bigr)$, 
a basis for $W_m=\bigwedge^{N_m} \HH^0\bigl(\PP^r, \O_{\PP^r}(m)\bigr)$ diagonalizing the $\rho$-action 
consists of $N_m$-tuples $e_1\wedge \ldots \wedge e_{N_m}$
of distinct monomials of degree $m$ in the variables $x_i$'s.
If $e_\ell=\prod_{i=0}^{r}x_i^{a_{\ell i}}$, then $\rho$ acts on $e_1 \wedge \ldots \wedge e_{N_m}$ with weight 
$\sum_{\ell=1}^{N_m}\sum_{i=0}^{r}a_{\ell i}\rho_i$.  
Now the condition that $\lim_{t \to 0} \rho(t) \cdot \widetilde{[X]}_m \neq 0$ is equivalent to 
the existence of one such coordinate which is non-vanishing on $[X]_m$ and on which $\rho$ acts with non-positive weight.
The condition that a coordinate $e_1 \wedge \ldots \wedge e_{N_m}$ is non-zero on $[X]_m$
is precisely the condition that the restrictions of $\{e_\ell\}_{\ell=1}^{N_m}$ to $X$ form a basis of $\HH^0\bigl(X, \O_{X}(m)\bigr)$. 
This discussion leads us to the following definition.
\begin{definition}
If $\{x_i\}_{i=0}^r$ is a $\rho$-weighted basis of $V$, 
a {\em monomial basis} of $\HH^0\bigl(X, \O_{X}(m)\bigr)$ is a set $\B=\{e_\ell\}_{\ell=1}^{N_m}$ of degree $m$ monomials
in the variables $\{x_i\}_{i=0}^r$
such that $\B$ 
maps onto a basis of $\HH^0 \bigl(X, \O_{X}(m) \bigr)$ 
via the restriction map $\HH^0 \bigl(\PP V, \O(m) \bigr) \rightarrow \HH^0 \bigl(X, \O_{X}(m) \bigr)$.  

Moreover, if $e_{\ell}=\prod_{i=0}^{r} x_i^{a_{\ell i}}$, 
we define the {\em $\rho$-weight} 
of $\B$ to be $w_{\rho}(\B):=\sum_{\ell=1}^{N_m}\sum_{i=0}^{r}a_{\ell i}\rho_i$. 
\end{definition}
With this terminology, we have the following criterion.
\begin{prop}[Numerical Criterion for Hilbert points]\label{P:hmc}
$[X]_m$ is semistable (resp., stable) if and only if for every $\rho$-weighted basis of $V$, 
there exists a monomial basis of $\HH^0\bigl(X, \O_{X}(m)\bigr)$ of non-positive (resp., negative) $\rho$-weight.
\end{prop}

The Hilbert-Mumford criterion reduces the problem of proving semistability of $[X]_m$ to a concrete algebro-combinatorial problem concerning 
the defining equations of $X \subset \PP V$. However, this problem is not discretely computable since it requires checking 
\emph{all} one-parameter subgroups of $\SL(V)$. A theorem of Kempf allows us, under certain hypotheses on $\Aut(X)$, to check only those 
one-parameter subgroups of $\SL(V)$ which act diagonally on a fixed basis.  This reduces the problem to one which is discretely computable.

In order to state the next proposition, let us establish a bit more terminology. 
Given an embedding $X\subset \PP V$ by a complete linear system, there is a natural 
action of $\Aut(X)$ on $V=\HH^0\bigl(X, \O_{X}(1)\bigr)$.  Given a linearly reductive subgroup $G\subset \Aut(X)$,
we say that $V$ is a \emph{multiplicity-free} $G$-representation (or simply {\em multiplicity-free} if $G$ is understood) 
if it contains no irreducible $G$-representation more than once in its decomposition into irreducible $G$-representations. 
We say that a basis of $V$, say $\{x_i\}_{i=0}^{r}$, 
is \emph{compatible with the irreducible decomposition of $V$} if each irreducible $G$-representation in $V$ is spanned 
by a subset of the $x_i$'s. We may now state the reformulation of Kempf's results that we will use. We keep the assumption that 
$X$ is embedded by a complete linear system $\vert \O_X(1)\vert$ and that the restriction map
$\HH^0\bigl(\PP V, \O(m)\bigr) \ra \HH^0\bigl(X,\O_X(m)\bigr)$
is surjective.
\begin{prop}[Kempf-Morrison Criterion]\label{P:kempf}
Suppose $G\subset \Aut(X)$ is a linearly reductive subgroup, 
and that $V=\HH^0\bigl(X,\O_X(1)\bigr)$ is a multiplicity-free representation of $G$. 
Let $\{x_i\}_{i=0}^{r}$ be a basis of $V$ which is compatible with the irreducible decomposition of $V$. 
Then $[X]_m$ is semistable (resp., stable) 
if and only if for every one-parameter subgroup $\rho\co \Spec \CC[t, t^{-1}] \rightarrow \SL(V)$ acting diagonally on $\{x_i\}_{i=0}^{r}$, 
we have $\lim_{t \rightarrow 0} \rho(t) \cdot \widetilde{[X]}_m \neq 0$ (resp., $\lim_{t \rightarrow 0} \rho(t) \cdot \widetilde{[X]}_m$ does not exist). 
Equivalently, for every weighted basis $\{x_i\}_{i=0}^{r}$ of $V$, there exists a monomial basis 
of $\HH^0\bigl(X, \O_{X}(m)\bigr)$ of non-positive (resp., negative) weight.
\end{prop}
\begin{proof}
If $[X]_m$ is not semistable, then \cite[Theorem 3.4 and Corollary 3.5]{kempf} implies that there is a 
one-parameter subgroup $\rho_*\co \Spec \CC[t, t^{-1}] \to \SL(V)$ with $\lim_{t \rightarrow 0} \rho_*(t) \cdot \widetilde{[X]}_m = 0$ 
such that the parabolic subgroup $P \subseteq \SL(V)$ associated to the $\rho_*$-weight filtration 
\[
0=U_0 \subseteq U_1 \subseteq \cdots \subseteq U_{k-1} \subseteq U_k = V
\]
contains $\Aut(X)$.  Let $V= \bigoplus_j V_j$ be the decomposition into irreducible $G$-representations.  
Since $V$ is multiplicity-free, each $U_i$ can be written as a direct sum of some of the $V_j$'s.  The maximal torus 
$T \subset \SL(V)$ associated to the basis $\{x_i\}_{i=0}^{r}$ fixes each $V_j$ and thus the filtration. Therefore, $T \subset P$.  
By \cite[Theorem 3.4 (c)(4)]{kempf}, there 
exists a one-parameter subgroup $\rho\co \Spec \CC[t, t^{-1}] \to T$ such that $\lim_{t \rightarrow 0} \rho(t) \cdot \widetilde{[X]}_m = 0$.
The statement for semistability follows.

The statement for {\em stability} follows by the same argument by replacing the concept of semistability
($0$-stability in Kempf's terminology) by a more general concept of $S$-stability; see \cite{kempf}. 
We are grateful to Ian Morrison for pointing this out.
\end{proof}

For the sake of concreteness, let us reiterate the Kempf-Morrison criterion in the case of a canonically (resp., bicanonically) 
embedded curve $C \subset \PP^r$. In order to prove that $[C]_{m}$ is semistable, we must first check that $V=\HH^0\bigl(C, K_{C}\bigr)$ 
(resp., $V=\HH^0\bigl(C, K_{C}^2\bigr)$) is a multiplicity-free representation of some linearly reductive $G\subset \Aut(C)$. 
Second, we fix a basis $\{x_i\}_{i=0}^{r}$ of $V$ compatible with the irreducible decomposition of $V$.
Now any one-parameter subgroup $\rho$ acting diagonally on $\{x_i\}_{i=0}^{r}$ is given by an integer weight vector 
$(\rho_0, \ldots, \rho_r)$ satisfying $\sum_{i=0}^{r}\rho_i=0$. To show that $[C]_m$ is semistable with respect to $\rho$, 
we must find a monomial basis $\B$ of $\HH^0\bigl(C,\O_C(m)\bigr)$ such that $w_{\rho}(\B)\leq 0$.
Note that for a fixed monomial basis $\B$, the $\rho$-weight function
$w_{\rho}(\B)$ is linear in $(\rho_0, \ldots, \rho_r)$.
Therefore, each monomial basis determines a half-space of weight vectors for which $[C]_m$ is 
$\rho$\nobreakdash-semistable, namely the half-space 
$w_{\rho}(\B) \leq 0$. It follows that as soon as one produces sufficiently many monomial bases 
such that the union of these half-spaces contains all weight vectors $(\rho_0, \ldots, \rho_r)$ satisfying 
$\sum_{i=0}^{r}\rho_i=0$, the proof of semistability for $[C]_m$ is completed. We summarize this discussion in the following lemma:
\begin{lemma}\label{L:covering} Let $G\subset \Aut(C)$ be a linearly reductive subgroup such that
$V=\HH^0\bigl(C,\O_C(1)\bigr)$ is a multiplicity-free representation of $G$, and let $\{x_i\}_{i=0}^{r}$ be a basis of $V$ 
which is compatible with the irreducible decomposition of $V$.
Suppose there exists a finite set $\{\B_j\}_{j\in J}$ of monomial bases of $\HH^0\bigl(C,\O_C(m)\bigr)$ and 
$\{c_j\}_{j\in J}\subset \mathbb{Q}\cap (0,\infty)$ such that 
\[
\sum_{j\in J} c_j w_{\rho}(\B_j)=0
\] 
for every $\rho\co \gm\ra \SL(V)$ acting on $\{x_i\}_{i=0}^r$ diagonally. 
Then $[C]_m$ is semistable.
\end{lemma}

The idea of applying these results of Kempf to the semistability of finite Hilbert points of curves 
is due to Morrison and Swinarski \cite{morrison-swinarski}. 
In their paper, they consider the so-called hyperelliptic {\em Wiman curve} $C$ with its bicanonical embedding.
They check that the 
automorphism group, which is cyclic of order $4g+2$, acts on $\HH^0\bigl(C, K_{C}^{2}\bigr)$ with $3g-3$ distinct characters. They fix a basis 
$\HH^0\bigl(C, K_{C}^2\bigr)=\{x_0, \ldots, x_r\}$ 
compatible with the decomposition of $\HH^0\bigl(C, K_{C}^2\bigr)$ into characters, and then, 
for low values of $g$ and $m$, use a computer to enumerate monomial bases of $\HH^0\bigl(C, \O_C(m)\bigr)$ 
until the associated half-spaces cover the hyperplane $\sum_{i=0}^{r}\rho_i=0$. 

In this paper, we apply the Kempf-Morrison criterion to canonically embedded 
ribbons of odd genus (Section \ref{S:monomial-bases-ribbon}), 
canonically embedded balanced double $A_{2k+1}$-curves of even genus (Section \ref{S:monomial-bases-A-curve}), 
bicanonically embedded rosaries of odd genus (Section \ref{S:monomial-bases-rosary}), 
and bicanonically embedded Wiman curves (Section \ref{S:monomial-bases-wiman}). 
For each $m \geq 2$, we write down \emph{by hand} sufficiently many monomial bases 
to establish the requisite (semi)stability result.

\section{Curves with $\gm$-action: Ribbons, $A_{2k+1}$-curves, and rosaries}\label{S:curves}
As discussed in the previous section, the key to our proof is to find a singular Gorenstein curve $C$ such 
that $\HH^0\bigl(C, \omega_C \bigr)$ (resp., $\HH^0\bigl(C, \omega_C^2 \bigr)$) is a multiplicity-free representation of 
$\Aut(C)$ in the canonical case (resp., bicanonical case).  In this section, we describe the curves we will 
use. In the odd genus canonical case, we will use a certain ribbon with $\GG_m$-action, the so-called 
{\em balanced ribbon}. In the even genus canonical case, we will use the 
{\em balanced double $A_{2k+1}$-curve}, 
i.e. a curve comprised of three $\PP^{1}$'s meeting in two higher tacnodes with trivial crimping. 
In the bicanonical case, we will use the so-called {\em rosary},
i.e. a cycle of $\PP^1$'s attached by tacnodes, introduced by Hassett and Hyeon in their classification of 
asymptotically stable bicanonical curves \cite{hassett-hyeon_flip}. 

A word of motivation as to where on earth these curves come from may be useful. That some class of 
canonically embedded ribbons should be GIT-semistable is intuitively plausible, since ribbons arise as  
flat limits of families of canonically embedded smooth curves degenerating abstractly to 
a hyperelliptic curve. The fact that the balanced ribbon of odd genus is the only ribbon with 
$\GG_m$-action that has the potential to be Hilbert semistable was proved in 
\cite[Theorem 7.2]{afs_preprint}. Hence, it was natural to attempt to prove that this curve is, in fact, semistable. 
Our motivation for considering double $A_{2k+1}$-curves comes from 
the log minimal model program for $\M_{2k}$, where we expect the $2k-4$ dimensional locus 
of double $A_{2k+1}$-curves to replace the 
locus in the boundary divisor $\Delta_{k}\subset \M_{2k}$ consisting 
of nodal curves $C_1\cup C_2$ such that each $C_i$ is a hyperelliptic curve of genus $k$. 
Indeed, this prediction has already been verified in $g=4$ by the second author who showed that 
the divisor $\Delta_2\subset \M_4$ is contracted to the point corresponding to the unique genus $4$
double $A_5$-curve in the final non-trivial log canonical model of $\M_4$ \cite{fedorchuk-genus-4}.
In the bicanonical case, we made use of the classification of asymptotically semistable curves in \cite{hassett-hyeon_flip}. 
We simply looked through the curves on their list for one with a large enough symmetry group to satisfy 
the hypotheses of Proposition \ref{P:kempf}. The rosary was the first curve we checked, and it worked!

\subsection{Canonical case, odd genus: The balanced ribbon with $\gm$-action}
\label{S:ribbon}
In this section we will construct, for every odd $g\geq 3$, a special non-reduced curve $C$ of arithmetic 
genus $g$ whose canonical embedding satisfies the hypotheses of Proposition \ref{P:kempf}. 
Given a positive odd integer $g=2k+1$, where $k\geq 1$, set $U:=\Spec \CC[u,\epsilon]/(\epsilon^2)$, $V:=\Spec \CC[v,\eta]/(\eta^2)$, 
and identify $U-\{0\}$ and $V-\{0\}$ via the isomorphism
\begin{align*}
u &\mapsto v^{-1}+v^{-k-2}\eta,\\
\epsilon &\mapsto v^{-g-1}\eta.
\end{align*}
The resulting scheme $C$ is evidently a complete, 
locally planar curve of arithmetic genus $g$; see \cite[Section 3]{BE} for more details on such curves. 
Note that $C$ admits $\GG_m$-action by the formulae 
\begin{align*}
t \cdot u&= tu, \\
t \cdot v&= t^{-1} v,\\
t \cdot \epsilon&= t^{k+1} \epsilon, \\
t \cdot \eta&=t^{-k-1} \eta. \\
\end{align*}
Since $C$ is locally planar, it is Gorenstein and its dualizing sheaf $\omega_C$ is a line bundle.
Using adjunction, we may identify global sections of $\omega_{C}$ with regular functions 
$f(u, \epsilon)$ on $U$. To be precise, the global sections of $\omega_{C}$ consist of all differentials 
\[
f(u, \epsilon) \frac{du \wedge d\epsilon}{\epsilon^2}
\]
which transform to differentials $h(v, \eta) \frac{dv \wedge d\eta}{\eta^2}$ with $h(v,\eta)$ regular on $V$. 
One easily writes down a basis of $g$ functions satisfying this condition to obtain the following lemma, which is a special
case of a more general \cite[Theorem 5.1]{BE}.
\begin{lemma}\label{L:ribbon-sections}
A basis for $\HH^0 \bigl(C, \omega_C \bigr)$ is given by differentials $f(u, \epsilon) \frac{du \wedge d\epsilon}{\epsilon^2}$ where $f(u, \epsilon)$ runs over the following list of $g$ functions:
\begin{equation*}
 x_i:=u^i, \quad 0\leq i \leq k, \qquad y_{k+i}:=u^{k+i}+iu^{i-1}\epsilon, \quad 1\leq i \leq k.
 \end{equation*}
 \end{lemma}

\begin{lemma}\label{L:ribbon-ample}
$\omega_C$ is very ample.
\end{lemma}
\begin{proof}
Using the basis of $\HH^0 \bigl(C,\omega_C \bigr)$ from Lemma \ref{L:ribbon-sections}, 
we see that $\vert \omega_C\vert$ separates points of 
$C_{\mathrm{red}}\simeq \PP^1$ and defines a closed embedding when restricted to $U$ and $V$. The claim follows.
\end{proof}
\begin{prop}\label{P:ribbon-multiplicity-free}
$\HH^0\bigl(C, \omega_{C} \bigr)$ is a multiplicity-free representation of $\gm\subset \Aut(C)$ and
$\{x_0,\dots, x_k, y_{k+1}, \dots, y_{2k}\}$ is compatible with its irreducible decomposition.
\end{prop}
\begin{proof}
The basis $\{x_0,\dots, x_k, y_{k+1}, \dots, y_{2k}\}$ diagonalizes the action of $\GG_m$ on $\HH^0 \bigl(C,\omega_{C} \bigr)$
with the $2k+1$ distinct weights $-k, \dots, -1, 0, 1, \dots, k$.
\end{proof}
In order to apply Proposition \ref{P:kempf}, we will need an effective way of determining when a set of 
monomials in the $g$ variables $\{x_0, \ldots, x_k, y_{k+1}, \ldots, y_{2k}\}$
forms a monomial basis of $\HH^0 \bigl(C, \omega_C^m \bigr)$. To do this, observe that the global sections 
of $\omega_C^m$ are easily identified with regular functions on $U$ via 
$f(u, \epsilon) \mapsto f(u, \epsilon) \dfrac{(du \wedge d\epsilon)^m}{\epsilon^{2m}}$. 
With this convention, we record the following observation used throughout the paper. 
\begin{lemma}[Ribbon Product Lemma]\label{L:RPL}
The expansion in $u$ and $\epsilon$ of the degree $m$ monomial $x_{i_1} \ldots x_{i_\ell}y_{i_{\ell+1}} \ldots y_{i_m}$ 
is $u^{a}+(a-b)u^{a-k-1}\epsilon$, where 
\begin{align*}
a&=i_1+\cdots +i_m,\\
b&=i_1+\cdots+i_\ell+k(m-\ell).
\end{align*}
\end{lemma}

The following 
proposition determines a basis for $\HH^0 \bigl(C, \omega_C^m \bigr)$ under the above identification.
\begin{prop}\label{P:projectively-normal-ribbon}  For $m \geq 2$, the product map
$\Sym^m \HH^0 \bigl(C, \omega_C \bigr) \to \HH^0 \bigl(C, \omega_C^{m} \bigr)$ is surjective.
A basis for $\HH^0 \bigl(C, \omega_C^m \bigr)$ is given by differentials 
$f(u, \epsilon) \frac{(du \wedge d\epsilon)^m}{\epsilon^{2m}}$ where $f(u, \epsilon)$ runs over 
the following $(2m-1)(g-1)$ functions on $U$:
\begin{align*}
\{u^i\}_{i=0}^{2mk-(k+1)}, \quad \{u^{i}+(i-k) u^{i-k-1} \epsilon\}_{i=k+1}^{2mk}.
\end{align*}
\end{prop}
\begin{proof}
We will show that the image of the product map 
$\Sym^m\HH^0 \bigl(C, \omega_C \bigr) \rightarrow \HH^0 \bigl(C, \omega_C^m \bigr)$ contains the given 
functions. Since $h^0(C, \omega_C^m)=(2m-1)(g-1)$ by Riemann-Roch, and because 
the given functions are linearly independent, this will prove the proposition.

Lemma \ref{L:RPL} gives $u^a=x_0^{m-1}x_a$ for $0\leq a\leq k$, 
$u^{2mk-k}+(2mk-2k)u^{2mk-2k-1}\epsilon=y_{2k}^{m-1}x_k$, 
and $u^{(2m-1)k+a}+\bigl((2m-2)k+a\bigr)u^{(2m-2)k+a-1}\epsilon=y_{2k}^{m-1}y_a$ for $1\leq a \leq k$.
For the intermediate $u$-degrees, note simply that since the dimension of the 
space 
$\{cu^i+d u^{i-k-1}\epsilon\,:\, c,d \in  \CC\}$ is two, we need to exhibit two linearly 
independent functions of this form as degree $m$ monomials in $\{x_0, \ldots, y_{2k}\}$. 
Using Lemma \ref{L:RPL}, this is an easy exercise which we leave to the reader.
\end{proof}

This result gives a very simple way of checking whether a set $\B$ of degree $m$ monomials in 
$\{x_0, \ldots, y_{2k}\}$ projects to a basis for 
$\HH^0\bigl(C,\O_C(m)\bigr)$. If we simply view the monomials in $\B$ as polynomials in $\CC[u,\epsilon]/(\epsilon^2)$ via the identification 
preceding Lemma \ref{L:RPL}, then $\B$ is a monomial basis of $\HH^0\bigl(C,\O_C(m)\bigr)$ 
 if and only if
\begin{enumerate}
\item $\B$ contains one polynomial of each $u$-degree $0, \ldots, k$,
\item $\B$ contains two linearly independent polynomials of each $u$-degree \\ $k+1, \ldots, (2m-1)k-1$,
\item $\B$ contains one polynomial of each  $u$-degree $2mk-k,\ldots, 2mk$. 
\end{enumerate}
We can rephrase this as follows.
\begin{lemma}\label{L:ribbon-basis}
A set of degree $m$ monomials 
\[
\{x_{i_1} \cdots x_{i_\ell}y_{i_{\ell+1}} \cdots y_{i_m}\}_{(i_1, \ldots, i_m) \in S}
\]
forms a monomial basis of $\HH^0\bigl(C,\O_C(m)\bigr)$ if and only if the following two conditions hold:
\begin{enumerate}
\item For $0 \leq a \leq k$ and $(2m-1)k \leq a \leq 2mk$, there is exactly one index vector $(i_1, \ldots, i_m) \in S$ with $i_1+\cdots+i_m=a$.
\item For $k < a <(2m-1)k$, there are exactly two index vectors $(i_1, \ldots, i_m) \in S$  satisfying $i_1+\cdots+i_m=a$. 
Furthermore, for these two index vectors, the associated integers $i_{\ell+1}+\cdots+i_m-k(m-\ell)$ are distinct.
 \end{enumerate}
 \end{lemma}
 \begin{proof}
Immediate from the preceding observations and the Ribbon Product Lemma \ref{L:RPL}.
\end{proof}

\subsection{Canonical case, even genus: The balanced double $A_{2k+1}$-curve with $\gm$-action}
\label{S:A-curves}
In this section we will construct special singular curves of even genus, whose canonical embeddings satisfy the hypotheses of Proposition \ref{P:kempf}. We define a {\em double $A_{2k+1}$-curve} to be any curve obtained by gluing three copies of $\PP^1$ along two
$A_{2k+1}$ singularities (Figure \ref{F:double-A-curve}). The arithmetic genus of a double $A_{2k+1}$-curve is $g=2k$, and double $A_{2k+1}$-curves have 
$2k-4$ moduli corresponding to the crimping of the $A_{2k+1}$-singularities, i.e. deformations that preserve the analytic types
of the singularities as well as the normalization of the curve (see \cite{fred} for a comprehensive treatment of crimping moduli). 
Indeed, the moduli space of crimping for an $A_{2k+1}$-singularity with automorphism-free branches 
has dimension $k$, but the presence of automorphisms of the pointed $\PP^1$'s in our situation
reduces the dimension of crimping moduli by $4$. Among double $A_{2k+1}$-curves, there is a unique double $A_{2k+1}$-curve with $\gm$-action, 
corresponding to the trivial choice of crimping for both $A_{2k+1}$-singularities. We call 
this curve the {\em balanced double $A_{2k+1}$-curve}.

Now let us give a more precise description of the balanced double $A_{2k+1}$-curve: Let $C_0, C_1, C_2$ denote three copies of 
$\PP^1$, and label the uniformizers at 0 (resp., at $\infty$) by $u_0, u_1,u_2$ (resp., by $v_0,v_1, v_2$). 
Fix an integer $k\geq 2$, and let $C$ be the arithmetic genus $g=2k$ curve obtained by gluing three $\PP^1$'s along two
$A_{2k+1}$ singularities with trivial crimping. More precisely, we impose an $A_{2k+1}$ singularity at 
$(\infty \in C_0) \sim (0 \in C_1)$ by 
gluing $C_0\setminus 0$ and $C_1\setminus \infty$ into an affine singular curve
\begin{equation}\label{E:gluing-1}
\Spec \CC[(v_0,u_1), (v_0^{k+1}, -u_1^{k+1})] \simeq \Spec \CC[x,y]/(y^2-x^{2k+2}).
\end{equation}
Similarly, we impose an $A_{2k+1}$ singularity at $(\infty \in C_1) \sim (0 \in C_2)$
by gluing $C_1\setminus 0$ and $C_2\setminus \infty$ into 
\begin{equation}\label{E:gluing-2}
\Spec \CC[(v_1,u_2), (v_1^{k+1}, -u_2^{k+1})] \simeq \Spec \CC[x,y]/(y^2-x^{2k+2}).
\end{equation}

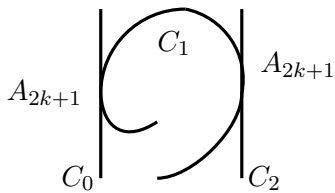
\begin{figure}[hbt]
\begin{centering}
\begin{tikzpicture}[scale=1.5]
		\node  (0) at (-3.5, 3) {};
		\node  (1) at (-2.75, 3) {};
		\node  (2) at (-2.25, 3) {};
		\node  (3) at (-1.75, 2.5) {$A_{2k+1}$};
		\node  (4) at (-3.5, 2.25) {};
		\node  (4x) at (-4, 2.25) {$A_{2k+1}$};
		\node  (5) at (-2.25, 2.25) {};
		\node  (6) at (-3, 2) {};
		\node  (7) at (-4.25, 1.5) {};
		\node  (8) at (-3.5, 1.5) {};
		\node  (8c) at (-3.7, 1.5) {$C_0$};
		\node  (9) at (-3, 1.5) {};
		\node  (10) at (-2.25, 1.5) {};
		\node  (10c) at (-2.05, 1.5) {$C_2$};
		\node  (C2) at (-2.85, 2.7) {$C_1$};
		\draw [very thick, bend left=45] (1.center) to (5.center);
		\draw [very thick] (2.center) to (10.center);
		\draw [very thick,bend left=300, looseness=1.50] (4.center) to (6.center);
		\draw [very thick] (0.center) to (8.center);
		\draw [very thick,bend left=45] (4.center) to (1.center);
		\draw [very thick,bend right=315, looseness=0.75] (5.center) to (9.center);
\end{tikzpicture}\end{centering}
\vspace{-0.5pc}
\caption{Double $A_{2k+1}$-curves}\label{F:double-A-curve}
\end{figure}

The automorphism group of $C$ is given by $\Aut(C) = \GG_m \rtimes \ZZ_2$ where 
$\ZZ_2$ acts via $u_i \leftrightarrow v_{2-i}$ and 
$\GG_m = \Spec \CC[t, t^{-1}]$ acts via 
\begin{align*}
t\cdot u_0 &= t u_0, \\
t\cdot u_1 &= t^{-1}u_1, \\
t\cdot u_2 &= t u_2.
\end{align*}

Using the description of the dualizing sheaf on a singular curve as in \cite[Ch.IV]{serre-corps} or \cite[Ch.II.6]{barth},
we can write down a basis of $\HH^0\bigl(C, \omega_C\bigr)$ as follows:
\begin{equation}\label{E:basis}
x_i =\left(u_0^{i}\, \frac{du_0}{u_0}, \, u_1^{-i}\, \frac{du_1}{u_1}, \, 0\right), 
\qquad y_i=\left(0,\, u_1^{i}\, \frac{du_1}{u_1},\, u_2^{-i}\, \frac{du_2}{u_2}  \right), \qquad 1\leq i \leq k.
\end{equation}

It is straightforward to generalize this description to the spaces of pluricanonical differentials.
\begin{lemma} 
\label{L:pluricanonical-bases}
For $m\geq 2$, the product map $\Sym^m \HH^0\bigl(C,\omega_C\bigr)\ra \HH^0\bigl(C, \omega_C^m\bigr)$ is surjective and
a basis of $\HH^0\bigl(C, \omega_C^m\bigr)$ consists of the following $(2m-1)(2k-1)$ differentials:
\begin{equation*}
\omega_j=\left(u_0^{j}\,\frac{(du_0)^m}{u_0^m}, \, u_1^{-j}\, \frac{(du_1)^m}{u_1^m},\, 0 \right), \quad 
\eta_j = \left(0, \, u_1^{j}\, \frac{(du_1)^m}{u_1^m}, \, u_2^{-j}\, \frac{(du_2)^m}{u_2^m} \right), \quad
m \leq j \leq mk.  
\end{equation*}
and
\begin{align*}
\chi_{\ell} &= \left(0, \, u_1^{\ell} \, \frac{(du_1)^m}{u_1^m}, \, 0 \right), \qquad -k(m-1)+1\leq \ell \leq k(m-1)-1.
\end{align*}
\end{lemma}
\begin{proof}
By Riemann-Roch formula, $h^0\bigl(C, \omega_C^m\bigr)=(2m-1)(2k-1)$. Thus, it suffices to observe that the given 
$(2m-1)(2k-1)$ differentials all lie in the image of the map 
$\Sym^m \HH^0\bigl(C,\omega_C\bigr) \rightarrow \HH^0\bigl(C, \omega_C^m\bigr)$. 
Using the basis of $\HH^0\bigl(C,\omega_C\bigr)$ given by \eqref{E:basis},
one easily checks that the differentials $\{\omega_j\}_{j=m}^{mk}$ are precisely those arising as $m$-fold products of 
$x_i$'s, the differentials $\{\eta_j\}_{j=m}^{mk}$ are those arising as $m$-fold products of $y_i$'s, and the differentials 
$\{\chi_\ell\}_{\ell=-k(m-1)+1}^{k(m-1)+1}$ are those arising as mixed $m$-fold products of $x_i$'s and $y_i$'s.
\end{proof}

Next, we show that $|\omega_C|$ is a very ample linear system, 
so that $C$ admits a canonical embedding, and the corresponding Hilbert points are well defined.

\begin{prop}\label{P:A-curve}
$\omega_C$ is very ample. The complete linear system $|\omega_C|$ embeds $C$ as a curve on a balanced rational normal scroll 
\[
\PP^1\times \PP^1 \stackrel{\vert \O(1, k-1)\vert}{\longhookrightarrow} \PP^{g-1}.
\]
Moreover, 
$C_0$ and $C_2$ map to $(1,0)$-curves on $\PP^1 \times \PP^1$, and $C_1$ maps to a $(1,k+1)$ curve. In particular, $C$ is a $(3,k+1)$ curve 
on $\PP^1\times \PP^1$ and has a $g^1_3$ cut out by the $(0,1)$ ruling.
\end{prop}

\begin{proof} 
To see that the canonical embedding of $C$ lies on a balanced rational normal scroll in $\PP^{2k-1}$, 
recall that the scroll 
is the determinantal variety (see \cite[Lecture 9]{harris}) defined by:
\begin{equation}\label{E:determinantal}
\rank
 \left(\begin{array}{cccc|cccc}
x_{1} & x_{2} & \cdots & x_{k-1} & y_{k} & y_{k-1} & \cdots & y_{2} \\  
 x_{2} & x_{3} & \cdots & x_{k} & y_{k-1} & y_{k-2} & \cdots & y_{1}
 \end{array}\right)\leq 1.
 \end{equation}
From our explicit description of the basis of $\HH^0\bigl(C, \omega_C\bigr)$ given by \eqref{E:basis}, 
one easily sees that the differentials $x_i$'s and $y_i$'s on $C$ 
satisfy the determinantal condition \eqref{E:determinantal}.
Moreover,
we see that $\vert \omega_C\vert$ embeds $C_0$ and $C_2$ as degree $k-1$ rational normal curves in $\PP^{2k-1}$ 
lying in the class $(1,0)$ on the scroll. Also, we see that $|\omega_C|$ embeds $C_1$ via the very ample linear 
system 
\[
\lspan\{ 1, u_1, \dots, u_1^{k-1}, u_1^{k+1}, \dots, u_1^{2k}\} \subset \vert \O_{\PP^1}(2k)\vert
\] 
as a curve in the class $(1,k+1)$. 
It follows that $|\omega_C|$ separates points and tangent vectors on each component of $C$. We now prove that $|\omega_C|$ 
separates points of different components and tangent vectors at the $A_{2k+1}$-singularities.  
First, observe that $C_0$ and $C_2$ span disjoint subspaces. Therefore, being $(1,0)$ 
curves, they must be distinct and non-intersecting. 
Second, $C_0$ and $C_1$ are the images of the two branches of an $A_{2k+1}$-singularity and so have
contact of order at least $k+1$. However, being $(1,0)$ and $(1,k+1)$ curves on the scroll, they have order of contact at most $k+1$. It follows
that the images of $C_0$ and $C_1$ on the scroll meet precisely in an $A_{2k+1}$-singularity. We conclude that $\vert \omega_C\vert$ is a closed embedding
at each $A_{2k+1}$-singularity. 

We can also directly verify that $\vert \omega_C\vert$ separates tangent vectors at an $A_{2k+1}$ singularity of $C$, 
say the one with uniformizers $v_0$ and $u_1$. The local generator of $\omega_C$ at this singularity is
\[
x_{k}=\left(-\frac{dv_0}{v_0^{k+1}}, \frac{du_1}{u_1^{k+1}}, 0 \right).
\] 
On the open affine chart $\Spec \CC[(v_0, u_1), (v_0^{k+1}, -u_1^{k+1})]$ defined in Equation \eqref{E:gluing-1},
we have $x_{k-1}=(v_0, u_1) \cdot x_{k}$ and $y_1=(0, u_1^{k+1}) \cdot x_{k}$. 
Under the identification $\CC[(v_0, u_1), (v_0^{k+1}, -u_1^{k+1})] \simeq \CC[x,y]/(y^2-x^{2k+2})$, 
we have $(v_0,u_1)=x$ and $(0,u_1^{k+1})=(x^{k+1}-y)/2$.
We conclude that $x_{k-1}$ and $y_1$ span the cotangent space, and thus separate tangent vectors, at the singularity. 
\end{proof}

Finally, the following elementary observation is the key to analyzing the stability of Hilbert points of $C$.

\begin{lemma}\label{L:multiplicityfree}  $\HH^0\bigl(C, \omega_C\bigr)$ is a multiplicity-free $\Aut(C)$-representation
and the basis $\{x_i, y_i\}_{i=1}^{k}$ is compatible with its irreducible decomposition.
\end{lemma}

\begin{proof} 
Note that $\GG_m\subset \Aut(C)$ acts on $x_i$ with weight $i$ and on $y_i$ with weight $-i$. Thus $\HH^0\bigl(C, \omega_C\bigr)$ 
decomposes into $g=2k$ distinct characters of $\GG_m$. 
\end{proof}

\subsection{Bicanonical case, odd genus: The rosary with $\gm$-action}
\label{S:rosary}
In this section we will construct, in every odd genus, a singular curve $C$ whose bicanonical embedding satisfies the hypotheses 
of Proposition \ref{P:kempf}. 
For any odd integer $g\geq 3$, we define $C$ to be the curve, called a {\em rosary} in \cite[Section 8.1]{hassett-hyeon_flip},
obtained from a set of $(g-1)$ $\PP^1$'s indexed by $i\in \ZZ_{g-1}$ and having uniformizers $u_i$ at $0$ and $v_i$ at $\infty$
(so that $u_i=1/v_i$) by cyclically identifying $v_i$ with $u_{i+1}$ to specify $g-1$ tacnodes. 
Note that $\GG_m \rtimes D_{g-1}\subset \Aut(C)$, where the dihedral group $D_{g-1}$ permutes the components 
and $\GG_m=\Spec \CC[t,t^{-1}]$ acts by $u_i \mapsto t^{(-1)^i} u_i$.  
We should remark that in the case of even genus, one may still define the curve $C$, but $C$ does not admit $\GG_m$-action and does not satisfy the hypotheses 
of Proposition \ref{P:kempf}. Thus, in what follows, we always assume $g$ odd.

\begin{lemma}\label{L:rosary-sections}
(a) A basis for $\HH^0\bigl(C,\omega_C\bigr)$ is given by the following differentials:
\begin{align*}
\omega_i&=\left(\dots, 0, du_i, \frac{du_{i+1}}{u_{i+1}^2}, 0, \ldots\right), \quad i\in \ZZ_{g-1}, \\
\eta&=\left(\frac{du_0}{u_0}, \frac{du_1}{u_1}, \dots, \frac{du_{g-2}}{u_{g-2}}\right).
\end{align*}
(b)  A basis for $\HH^0\bigl(C,\omega_C^2 \bigr)$ is given by the following differentials: 
\begin{align*}
x_i&=\omega_i^2, \qquad i\in \ZZ_{g-1}, \\
y_i&=\omega_i\eta, \qquad i\in \ZZ_{g-1}, \\
z_i&=\omega_{i-1}\omega_{i}, \quad  i\in \ZZ_{g-1}.
\end{align*}
\end{lemma}
\begin{proof} Using duality on singular curves as in \cite[Ch.IV]{serre-corps} or \cite[Ch.II.6]{barth},
it is straightforward to verify that each differential from (a) is a Rosenlicht differential and 
hence is an element of $\HH^0 \bigl(C,\omega_C \bigr)$. Since these $g$ differentials are linearly 
independent, Part (a) is established. Part (b) follows immediately: 
The $(3g-3)$ differentials from (b) are products of elements in $\HH^0\bigl(C,\omega_C\bigr)$
and are easily seen to be linearly independent.
\end{proof}

\begin{lemma}\label{L:rosary-ample}
$\omega_C$ is very ample for odd $g\geq 5$ and $\omega_C^2$ is very ample for odd $g\geq 3$.
\end{lemma}
\begin{proof}
We prove that $C$ is canonically embedded for $g\geq 5$. 
First, observe that $\vert \omega_C\vert$ embeds 
each $\PP^1$ as a conic in $\PP^{g-1}$, and that the plane spanned by the $i^{th}$ conic meets 
only the planes spanned by the cyclically adjacent conics, and meets each of these
only at the corresponding tacnode. This shows that $\vert \omega_C\vert$ separates points and tangent vectors
at smooth points. To see
that $\vert \omega_C\vert$ separates tangent vectors at the tacnode obtained by the identification $v_i=u_{i+1}$, note that 
the local generator of $\omega_C$ at this tacnode is $\omega_i$. Locally around the 
tacnode, we have $\eta=(v_i, u_{i+1})\cdot \omega_i$ and $\omega_{i+1}=(0, u_{i+1}^2)\cdot \omega_i$. Under the identification 
$\CC[(v_i, u_{i+1}),(0, u_{i+1}^2)]\simeq \CC[x,y]/\bigl(y(x^2-y)\bigr)$, we have $(v_i, u_{i+1})=x$ and $(0, u_{i+1}^2)=y$. 
We conclude that $\eta$ and $\omega_{i+1}$ span the cotangent space, and thus separate tangent vectors, at the tacnode. 

A straightforward computation shows that $\omega_C^2$ is also very ample for $g=3$. We finish by noting
that $C$ is hyperelliptic in genus $3$ and thus is not canonically embedded.
\end{proof} 
The $\GG_m$-action on $\HH^0\bigl(C, \omega_C\bigr)$ is given by 
\begin{align*}
t\cdot \omega_i &=t^{(-1)^i}\omega_i, 
\\
t\cdot \eta &= \eta.
\end{align*}
The $\GG_m$-action on $\HH^0\bigl(C, \omega_C^{2}\bigr)$ is given by $x_i \mapsto (t^2)^{(-1)^i}x_i$, 
$y_i \mapsto t^{(-1)^i}y_i$, $z_i \mapsto z_i$.  
We define the weight of a monomial to be its $\GG_m$-weight. 
\begin{prop}\label{P:rosary-multiplicity-free} Both
$\HH^0\bigl(C, \omega_C\bigr)$ and $\HH^0\bigl(C, \omega_{C}^2\bigr)$ 
are multiplicity-free representations of $\gm\rtimes \ZZ_{g-1}\subset \Aut(C)$. 
Moreover, the basis $\{\omega_0, \dots, \omega_{g-2}, \eta\}$ is compatible with the irreducible decomposition 
of $\HH^0\bigl(C, \omega_C\bigr)$,
and the basis $\{x_i, \, y_i,\, z_i \ : \ i\in \ZZ_{g-1}\}$  is compatible with the irreducible decomposition of $\HH^0\bigl(C, \omega_{C}^2\bigr)$.

\end{prop}
\begin{proof}
 The action of $\ZZ_{g-1} \subset D_{g-1}$ on the span of $\left\{ \omega_i \right\}_{i=0}^{g-2}$ 
 (resp.,  $\left\{ x_i \right\}_{i=0}^{g-2}$, $\left\{y_i \right\}_{i=0}^{g-2}$, 
 $\left\{ z_i \right\}_{i=0}^{g-2}$) corresponds to the regular representation of $\ZZ_{g-1}$ 
 and is thus multiplicity-free.  Since the weight of 
 $\omega_i$ is $\pm 1$ and of $\eta$ is $0$ (resp., the weight of $x_i$ is $\pm 2$, of $y_i$ is $\pm 1$, and of $z_i$ is $0$), 
 it follows that $\HH^0\bigl(C, \omega_{C}\bigr)$ (resp., $\HH^0\bigl(C, \omega_{C}^2\bigr)$) is a multiplicity-free representation
 of $\GG_m \rtimes \ZZ_{g-1}$. 
\end{proof}

The following lemmas are elementary and so we omit the proofs.
\begin{lemma}
\label{L:rosary-basis-canonical} The multiplication map 
$\Sym^m \HH^0\bigl(C, \omega_C \bigr) \to \HH^0 \bigl(C, \omega_C^{m} \bigr) $ is surjective. 
A set $\B$ of degree $m$ monomials in $\omega_0,\dots, \omega_{g-2},\eta$ forms a monomial basis of $\HH^0 \bigl(C,\omega_C^m \bigr)$
if and only if the following conditions are satisfied:
\begin{enumerate}
\item $\B$ contains the $(g-1)$ monomials $\{\omega_i ^m\}_{i=0}^{g-2}$ of weight $\pm m$,
\item $\B$ contains the $(g-1)$ monomials $\{\omega_i^{m-1}\eta\}_{i=0}^{g-2}$ of weight $\pm (m-1)$,
\item $\B$ contains $(g-1)$ linearly independent monomials of each weight $2-m \leq j \leq m-2$.
\end{enumerate}
\end{lemma}
The reader may wish to check, as an example, that
$\{ \omega_i^{j}\eta^{m-j}\}_{i=0}^{g-2}$ and $\{ \omega_{i}^{j+1}\omega_{i-1}\eta^{m-j-2}\}_{i=0}^{g-2}$
give $2g-2$ linearly independent monomials, with $(g-1)$ monomials of weights $j$ and $-j$ each. Thus, taking the union of all these monomials, 
together with $\{\omega_i ^m\}_{i=0}^{g-2}$ and  $\{\omega_i^{m-1}\eta\}_{i=0}^{g-2}$ gives a monomial basis of $\HH^0\bigl(C,\omega_C^m\bigr)$. 

\begin{lemma} \label{L:rosary-basis-bicanonical}
The multiplication map $\Sym^m \HH^0\bigl(C, \omega_C^2\bigr) \to \HH^0\bigl(C, \omega_C^{2m}\bigr)$ is surjective.
A set $\B$ of degree $m$ monomials in $\{x_i\}_{i=0}^{g-2}$, $\{y_i\}_{i=0}^{g-2}$, 
$\{z_i\}_{i=0}^{g-2}$ forms a monomial basis of $\HH^0\bigl(C,\omega_C^{2m}\bigr)$ if and only if the following conditions are satisfied:
\begin{enumerate}
\item $\B$ contains the $(g-1)$ monomials $\{x_i ^m\}_{i=0}^{g-2}$ of weight $\pm 2m$,
\item $\B$ contains the $(g-1)$ monomials $\{x_i^{m-1}y\}_{i=0}^{g-2}$ of weight $\pm (2m-1)$,
\item $\B$ contains $(g-1)$ linearly independent monomials of each weight $2-2m \le j \le 2m-2$. 
\end{enumerate}
\end{lemma}

\section{Monomial bases and semistability} \label{S:monomial-bases}

\subsection{Canonically embedded ribbon}\label{S:monomial-bases-ribbon}

Let $C$ denote the balanced ribbon as defined in Section \ref{S:ribbon}. In this section, we prove the odd genus case 
of the first part of our Main Result.

\begin{theorem}\label{T:ribbon}
If $C \subset \PP \HH^0\bigl(C, \omega_C\bigr)$ is a canonically embedded balanced ribbon, 
then the Hilbert points $[C]_m$ are semistable for all $m \geq 2$.
\end{theorem}

\begin{corollary}\label{C:ribbon}
Suppose $C \subset \PP \HH^0\bigl(C, K_{C}\bigr)$ is a canonically embedded generic smooth curve of odd genus. 
Then the $m^{th}$ Hilbert point of $C$ is semistable for every $m \geq 2$.
\end{corollary}

\begin{proof}[Proof of Corollary \ref{C:ribbon}]
Quite generally, the locus of semistable points $(\overline{H}_{g,1}^{\, m})^{ss} \subset \overline{H}_{g,1}^{\, m}$ is open \cite{git}. 
Since $\overline{H}_{g,1}^{\, m}$ is an irreducible variety whose generic point is the $m^{th}$ Hilbert point of a canonically embedded 
smooth genus $g$ curve, it remains to find a single semistable point in $\overline{H}_{g,1}^{\, m}$.
The balanced ribbon $C$ deforms to a smooth canonical curve by \cite{fong} and 
Proposition \ref{P:projectively-normal-ribbon} shows that $[C]_{m}\in \overline{H}_{g,1}^{\, m}$.
Applying Theorem \ref{T:ribbon} finishes the proof.
\end{proof}

We have already seen that there is a distinguished basis $\{x_0,\dots, x_k, y_{k+1},\dots, y_{2k}\}$ 
of $\HH^0\bigl(C, \omega_C\bigr)$ on which $\gm\subset \Aut(C)$ acts with 
distinct weights (Proposition \ref{P:ribbon-multiplicity-free}). 
According to Lemma \ref{L:covering}, to prove Theorem \ref{T:ribbon}
it suffices to find a set of monomial bases such that an effective linear combination of their $\rho$-weights is $0$ with respect 
to every one-parameter subgroup $\rho\co \gm \ra \SL(g)$.
For ease of exposition, we will treat the cases $m=2$ and $m \geq 3$ separately.

\subsubsection{Monomial bases of $\HH^0\bigl(C,\omega_C^2\bigr)$.}\label{S:ribbon:m=2}

First, we define two monomial bases, $\B^{+}$ and $\B^{-}$, of $\HH^0\bigl(C, \omega_C^2\bigr)$ as follows. 
We define $\B^{+}$ to be the set of quadratic monomials divisible by one of $x_0$, $x_k$, or $y_{2k}$. 
More precisely, 
\begin{equation}\label{E:B+}
\B^{+}:=\left\{
 \{x_0x_i\}_{i=0}^{k}, \{x_0y_i\}_{i=k+1}^{2k}, 
 \{x_kx_i\}_{i=1}^{k}, \{x_ky_i\}_{i=k+1}^{2k},  
 \{y_{2k}x_i\}_{i=1}^{k-1}, \{y_{2k}y_i\}_{i=k+1}^{2k}
\right\}.
\end{equation}
We define $\B^{-}$ as follows:
\begin{equation}\label{E:B-}
\B^{-}:=\left\{
\begin{aligned}
&\{x_{i}^2\}_{i=0}^{k},  \{y_{i}^2\}_{i=k+1}^{2k},  \\  
& \{x_{i}x_{i+1}\}_{i=0}^{k-1}, x_ky_{k+1}, \{y_iy_{i+1}\}_{i=k+1}^{2k-1}, \\ 
&\{x_iy_{i+k}\}_{i=1}^{k-1}, \{x_iy_{i+k+1}\}_{i=0}^{k-1}
\end{aligned}\right\}.
\end{equation}

\begin{lemma}
$\B^{+}$ and $\B^{-}$ are monomial bases of $\HH^0\bigl(C, \omega_C^2\bigr)$. For any 
one-parameter subgroup $\rho$ acting on $(x_0,\dots, y_{2k})$ diagonally with weights
 $(\rho_0, \ldots, \rho_{2k})$ the $\rho$-weights of $\B^+$ and $\B^-$ are:
\begin{align*}
&w_{\rho}(\B^{+})=(g-2)(\rho_0+\rho_k+\rho_{2k}),\\
&w_{\rho}(\B^{-})=-2(\rho_0+\rho_k+\rho_{2k}).
\end{align*}
\end{lemma}

\begin{proof} Using Lemma \ref{L:ribbon-basis}, one easily checks that $\B^+$ and $\B^-$ are monomial bases. 
To compute the weight of $\B^{+}$ observe that variables $\{x_i, y_{k+i}\}_{i=1}^{k-1}$ each occur $3$ times and variables
$\{x_0, x_k, y_{2k}\}$ each occur $g+1$ times in Display \eqref{E:B+}. It follows that 
\[
w_{\rho}(\B^{+})=3\sum_{i=1}^{k-1} (\rho_i+\rho_{k+i})+(g+1)(\rho_0+\rho_k+\rho_{2k})=(g-2)(\rho_0+\rho_k+\rho_{2k}),
\]
where the last equality follows from the relation $\sum_{i=0}^{2k}\rho_i=0$.

Similarly, variables $\{x_i, y_{k+i}\}_{i=1}^{k-1}$ each occur $6$ times and variables $\{x_0, x_k, y_{2k}\}$ each occur $4$ times 
in Display \eqref{E:B-}. It follows that 
\[
w_{\rho}(\B^{-})=6\sum_{i=1}^{k-1} (\rho_i+\rho_{k+i})+4(\rho_0+\rho_k+\rho_{2k})=-2(\rho_0+\rho_k+\rho_{2k}),
\]
where the last equality again follows from the relation $\sum_{i=0}^{2k}\rho_i=0$.
\end{proof}

\begin{corollary} \label{C:ribbon2}
The $2^{nd}$ Hilbert point of $C$ is semistable.
\end{corollary}

\begin{proof}
We have $2w_{\rho}(\B^+)+(g-2)w_{\rho}(\B^-)=0$ for any $\rho\co \gm\ra \SL(g)$ acting diagonally on the distinguished basis. 
The claim follows by Lemma \ref{L:covering}.
\end{proof}

\subsubsection{Monomial bases of $\HH^0\bigl(C,\omega_C^m\bigr)$ for $m\geq 3$.}
Finding monomial bases in higher degrees is slightly more cumbersome than in the case $m=2$. First, 
we will need three monomial bases in every degree $m \geq 3$. Second, 
the precise form of one of these bases depends on the residue of $g=2k+1$ modulo $4$. 
Nevertheless, the proof is conceptually no different than in the case $m=2$. 
Finally, we work throughout with $m$ fixed and each basis used in degree $m$ is defined independently as a set of degree $m$ monomials,
though we have, for simplicity, suppressed the dependence on $m$ in our notation.

We begin by defining two higher-degree analogues
of the basis $\B^{+}$ from Section \ref{S:ribbon:m=2}. 
\begin{definition}\label{D:Petri-1} 
We define $\B^{+}_{1}$ to be the set of degree $m$ monomials in the ideal 
\[
(x_0,x_k)^{m-1}\cdot (x_0,\dots, x_{k-1}, y_{k+1}, \dots, y_{2k})
+(x_k,y_{2k})^{m-1}\cdot (x_0,\dots, x_{k-1}, y_{k+1}, \dots, y_{2k})+x_k^m.
\]
We define $\B^{+}_2$ to be the set of degree $m$ monomials in the ideal 
 \begin{multline*}
 (x_0,y_{2k})^{m-1}\cdot (x_1,\dots, x_{k-1}, y_{k+1}, \dots, y_{2k-1}) +x_k\cdot (x_0,y_{2k})^{m-2}\cdot (x_1,\dots, x_{k-1}, y_{k+1}, \dots, y_{2k-1}) \\
 +(x_0,y_{2k})^{m}+x_k(x_0,y_{2k})^{m-1}+x^2_k(x_0,y_{2k})^{m-2}+x^3_k(x_0,y_{2k})^{m-3} 
 \end{multline*}
\end{definition}
\begin{lemma}\label{L:Petri-weight-1}
$\B^{+}_1$ and $\B^{+}_{2}$ are monomial bases of $\HH^0\bigl(C, \omega_C^m\bigr)$. 
For any one-parameter subgroup $\rho$ acting on $(x_0,\dots, y_{2k})$ diagonally with weights
 $(\rho_0, \ldots, \rho_{2k})$ the $\rho$-weights of $\B^{+}_1$ and $\B^{+}_2$ are:
\begin{align*}
w_{\rho}(\B_1^{+})=\bigl((m-1)^2(g-1)-(2m-3)\bigr)\rho_k+\left(\frac{m(m-1)}{2}(g-1)-1\right)(\rho_0+\rho_{2k}), \\
w_{\rho}(\B_2^{+})=\bigl((m-1)(g-1)+(2m-5)\bigr)\rho_k+\bigl((m-1)^2(g-1)-(2m-3)\bigr)(\rho_0+\rho_{2k}).
\end{align*}
\end{lemma}
\begin{proof}
Using Lemma \ref{L:ribbon-basis}, it is easy to see that $\B^+_1$ and $\B^+_2$ are monomial bases. 
Next, note that in $\B_1^{+}$ the variable $x_k$ appears $(m-1)^2(g-1)+2$ times, 
variables $x_0$ and $y_{2k}$ each appear $\binom{m}{2}(g-1)+2m-2$ times,
and variables $x_1,\dots, x_{k-1}, y_{k+1},\dots, y_{2k-1}$
each appear $2m-1$ times. Recalling that $\sum_{i=0}^{2k} \rho_i=0$, we deduce the
formula for $w_{\rho}(\B_1^{+})$.

The $\rho$-weight of $\B_2^{+}$ is computed analogously by observing that 
in $\B_2^{+}$ the variable $x_k$ appears $(m-1)(g-1)+(4m-6)$ times, variables $x_0$ and $y_{2k}$ each appear $(m-1)^2(g-1)+2$ times, 
and variables $x_1,\dots, x_{k-1}$ and $y_{k+1},\dots, y_{2k-1}$ each appear $2m-1$ times.
\end{proof}

Next, we construct higher-degree analogues of the basis $\B^-$ from Section \ref{S:ribbon:m=2}. 
Throughout the construction, we let $\iota$ be the involution exchanging $x_i$ and $y_{2k-i}$ and leaving $x_k$ fixed. 

Let $\ell=\lfloor k/2 \rfloor$. We introduce the following sets of monomials:
\begin{align*}
S_0&:=
\begin{cases} \left\{x_k^m, \ x_0\, y_{2k}\, x_k^{m-2}\right\} & \text{if $m$ is odd}, \\ 
\left\{x_k^m, \ x_\ell\, y_{2k-\ell}\, x_k^{m-2}\right\} & \text{if $m$ is even.} \end{cases}\\
S_1 &:=\left\{
\begin{aligned}
& x_i^{m-d}\, x_{i+1}^{d}\ : \quad 0\leq i\leq k-1,\ 0\leq d\leq m-1 \\
\end{aligned}\right\} \\
S_2 &:=\left\{
\begin{aligned}
&x_i^{m-1-d} \,x_{i+1}^{d}\,y_{i+k+1}\ : \quad 0 \leq i \leq \ell-2, \ 0\leq d \leq m-1 \\
&x_{\ell-1}^{m-1-d} \,x_{\ell}^{d}\, y_{\ell+k}\ : \qquad 0\leq d \leq m-2 \end{aligned}\right\}  
\end{align*}
The definition of the next set of monomials depends on parity of $k$. If $k=2\ell$, we define
\begin{align*}
S_3 &:=\left\{ x_i^{m-1-d} \,x_{i+1}^{d}\,y_{k+2\ell-1-i}\ : \quad \ell \leq i\leq k-2, \ 1\leq d\leq m-2 \right\};
\end{align*}
and if $k=2\ell+1$, we define
\begin{align*}
S_3 &:=\left\{
\begin{aligned}
&x_i^{m-1-d} \,x_{i+1}^{d}\,y_{k+2\ell-1-i}\ : \quad \ell \leq i\leq k-3, \ 1\leq d\leq m-2 \\
&x_{k-2}^{m-2-d} \,x_{k-1}^{d}\,x_{\ell}\,y_{3\ell+1} \ : \quad 0\leq d\leq m-2 \end{aligned}\right\}  
\end{align*}
We proceed to define
\begin{align*}
S_4 &:=\left\{x_{k-1}^{m-2-d} \,x_{k}^{d}\,(x_0y_{2k})\ : \quad 0\leq d\leq m-4 \right\} \\
S_5 &:=\begin{cases} \left\{x_{k-1}\,x_\ell^{(m-1)/2}\, y_{2k-\ell}^{(m-1)/2}\right\} & \text{if $m$ is odd}, \\ \\
\left\{x_{k-1}\,x_k\,x_\ell^{(m-2)/2}\,y_{2k-\ell}^{(m-2)/2}\right\} & \text{if $m$ is even.} \end{cases} 
\end{align*}

\begin{definition}\label{D:anti-Petri-1} 
We define a set $\B^{-}$ of degree $m$ monomials by 
\[
\B^{-}:= S_0\cup \bigcup_{i=1}^5 \bigl(S_i \cup \iota(S_i)\bigr).
\]
\end{definition}

\begin{lemma}\label{L:weight-anti-Petri-1} For $m\geq 3$,
$\B^{-}$ is a monomial basis of $\HH^0\bigl(C,\O_C(m)\bigr)=\HH^0\bigl(C,\omega_C^m\bigr)$. For any $\rho\co \gm \ra \SL(g)$ 
acting on $(x_0,\dots, y_{2k})$ diagonally with weights
 $(\rho_0, \ldots, \rho_{2k})$ we have 
\begin{align*}
w_{\rho}(\B^{-})=\begin{cases} -(m^2-3m+5)(\rho_0+\rho_{2k})-(5m-10)\rho_{k} & \text{if $m$ is odd}, \\
-(m^2-3m+6)(\rho_0+\rho_{2k})-(5m-12)\rho_{k} & \text{if $m$ is even}. \end{cases}  
\end{align*}
\end{lemma}

\begin{proof} Although the precise definition of $\B^{-}$ depends on the parity of $k$, our proof of the lemma 
does not. Thus we suppress the parity of $k$ in what follows. To prove that $\B^{-}$ is a monomial 
basis, we make use of the identification of $\HH^0\bigl(C,\omega_C^m\bigr)$ with functions in $\CC[u,\epsilon]/(\epsilon^2)$ made in
Section \ref{S:ribbon}. 
To begin, observe that $\B^{-}$ is invariant under $\iota$. Since $\iota$ maps a monomial of $u$-degree $d$ to 
a monomial of $u$-degree $2mk-d$, it suffices, in view of Lemma \ref{L:ribbon-basis}, to show that $\B^{-}$ 
contains one monomial of each $u$-degree $d=0,\dots, k$ and two linearly independent monomials of 
each $u$-degree $d=k+1,\dots, mk$.
To do this, note that $S_0$ consists of 
two linearly independent monomials of $u$-degree $km$; that
$S_1$ consists by the Ribbon Product Lemma \ref{L:RPL} 
of exactly pure powers of $u$ of each $u$-degree $d=0,\dots, mk-1$;
and that $S_2 \cup S_3 \cup S_4\cup S_5$ contains exactly one monomial
of each $u$-degree $d=k+1,\dots, mk-1$ with a non-zero $\epsilon$ term. 
This finishes the proof that $\B^{-}$ is a monomial basis.

To compute the $\rho$-weight of $\B^{-}$, we observe that in $S_1\cup \{x_k^m\} \cup \iota(S_1)$ all variables with the exception 
of $x_0$ and $y_{2k}$ occur the same number of times, namely $2\sum_{d=1}^{m-1}d+m=m^2$ times, while $x_0$ and $y_{2k}$ each 
occur $\sum_{d=1}^{m}d=m(m+1)/2$ times. It follows that the $\rho$-weight of $S_1\cup \{x_k^m\} \cup \iota(S_1)$
is 
\[
m^2\sum_{i=1}^{2k-1}\rho_i+\frac{m(m+1)}{2}(\rho_0+\rho_{2k})=-m(m-1)(\rho_0+\rho_{2k})/2,
\]
where the last equality follows from $\sum_{i=0}^{2k} \rho_i=0$. 
Similarly, one can easily see that in the remaining monomials of $\B^{-}$  
each of the variables
$x_1,\dots, x_{k-1}, y_{k+1},\dots, y_{2k}$ occurs exactly $m(m-1)$ times; each of the variables $x_0$ and $y_{2k}$ 
occurs 
\begin{align*}
\begin{cases} (m^2+3m-10)/2 
& \text{(if $m$ is odd)} \\
(m^2+3m-12)/2 
& \text{(if $m$ is even)}
\end{cases} \quad \text{times};
\end{align*}
and $x_k$ occurs 
\begin{align*}
\begin{cases} m^2-6m+10 
& (\text{if $m$ is odd}) \\
m^2-6m+12  
& (\text{if $m$ is even}) 
\end{cases} \quad \text{times}.
\end{align*}

Using $\sum_{i=0}^{2k} \rho_i=0$, it follows that the total $\rho$-weight of 
these remaining monomials is
\begin{align*}
&-\frac{(m^2-5m+10)}{2}(\rho_0+\rho_{2k})-(5m-10)\rho_k \quad  \text{if $m$ is odd},  \\
&-\frac{(m^2-5m+12)}{2}(\rho_0 + \rho_{2k})-(5m-12)\rho_k\ \quad \text{if $m$ is even}.
\end{align*}
The claim follows.
\end{proof}

\begin{lemma}\label{L:ribbon-linear-dependence} 
There exist $c_0, c_1, c_2\in \mathbb{Q}\cap (0,\infty)$ such that 
\[
c_0w_{\rho}(\B^{-})+c_1w_{\rho}(\B^{+}_1)+c_2w_{\rho}(\B^{+}_2)=0
\]
for all one-parameter subgroups of $\SL(g)$
acting on the basis $(x_0,\dots, y_{2k})$ diagonally.
\end{lemma}

\begin{proof}
We need to show that $w_{\rho}(\B^{-})$ given by Lemma \ref{L:weight-anti-Petri-1} and considered 
as the linear function in $(\rho_0,\dots, \rho_{2k})$ is the negative of an effective linear combination 
of $w_{\rho}(\B^{+}_1)$ and $w_{\rho}(\B^{+}_2)$ given by Lemma \ref{L:Petri-weight-1}.
In the case of odd $m$, the claim holds because the inequalities
\begin{equation*}
\frac{m(m-1)(g-1)-2}{2(m-1)^2 (g-1)-2(2m-3)} \leq \frac{(m^2-3m+5)}{5m-10} 
\leq \frac{(m-1)^2 (g-1)-(2m-3)}{(m-1)(g-1)+(2m-5)}, 
\end{equation*}
are satisfied for all $g, m\geq 3$. 
In the case of even $m$, we require the same inequalities save that the middle term is replaced by $\frac{(m^2-3m+6)}{5m-12}$.
\end{proof}

\begin{proof}[Proof of Theorem \ref{T:ribbon}.] 
The case of $m=2$ was handled in Corollary \ref{C:ribbon2}.  If $m \geq 3$, 
the claim follows from Lemma \ref{L:ribbon-linear-dependence} and Lemma \ref{L:covering}.
\end{proof}

\subsection{Canonically embedded $A_{2k+1}$-curve} \label{S:monomial-bases-A-curve}

Let $C$ denote the balanced double $A_{2k+1}$-curve as defined in Section \ref{S:A-curves}.
In this section, we prove the even genus case of the first part of our Main Result. 
Since $\HH^0\bigl(C,\omega_C\bigr)$ is a multiplicity-free representation of $\gm \subset \Aut(C)$
by Lemma \ref{L:multiplicityfree}, we can apply the Kempf-Morrison Criterion (Proposition \ref{P:kempf}) to prove semistability of $C$.
Namely, to prove that $[C]_m$ is semistable, it suffices to check that 
for every one-parameter subgroup 
$\rho\co \gm \rightarrow \SL(g)$ acting diagonally on the distinguished basis 
$\{x_1, \ldots, x_{k}, y_1, \ldots, y_{k}\}$ with integer weights $\lambda_1, \ldots, \lambda_{k}, \nu_1, \ldots, \nu_{k}$, 
there exists a monomial basis of $\HH^0\bigl(C, \omega_C^m\bigr)$ of non-positive $\rho$-weight.  
Explicitly, this means that we must exhibit a set $\B$ of $(2m-1)(2k-1)$ 
degree $m$ monomials in the variables $\{x_i,y_i\}_{i=1}^{k}$ 
with the properties that:
\begin{enumerate}
\item $\B$ maps to a basis of $\HH^0\bigl(C, \omega_C^m\bigr)$ 
via $\Sym^m \HH^0\bigl(C,\omega_C\bigr) \rightarrow \HH^0\bigl(C, \omega_C^m\bigr)$.
\item $\B$ has non-positive $\rho$-weight, that is, if $\B=\{e_\ell\}_{\ell=1}^{(2m-1)(2k-1)}$ and 
$e_\ell=\prod_{i=1}^{k}x_i^{a_{\ell i}}y_i^{b_{\ell i}}$, then 
\[
\sum_{\ell=1}^{(2m-1)(2k-1)} \sum_{i=1}^k (a_{\ell i}\lambda_i+b_{\ell i}\nu_i) \leq 0.
\]
\end{enumerate}

\begin{theorem}\label{T:double-A-curve}
If $C \subset \PP \HH^0\bigl(C, \omega_C\bigr)$ is a canonically embedded balanced double $A_{2k+1}$-curve, then the Hilbert 
points $[C]_m$ are semistable for all $m \geq 2$.
\end{theorem}
As an immediate corollary of this result, we obtain a proof of Theorem \ref{T:main-trigonal} (1)
and hence of Theorem \ref{T:main}:
\begin{corollary}[Theorem \ref{T:main-trigonal} (1)]\label{C:main-trigonal}
A generic canonically embedded smooth trigonal curve of even genus 
has semistable $m^{th}$ Hilbert point for every $m \geq 2$.
\end{corollary}

\begin{proof}[Proof of Corollary]
Recall from Proposition \ref{P:A-curve} that 
the canonical embedding of the balanced double $A_{2k+1}$-curve $C$ lies on a balanced surface scroll in $\PP^{2k-1}$
in the divisor class $(3, k+1)$.
It follows that $C$ deforms flatly to a smooth curve in the class $(3, k+1)$ on the scroll. Such a curve is a canonically embedded 
smooth trigonal curve of genus $2k$.
The semistability of a generic deformation of $C$ follows from the openness of semistable locus. 
\end{proof}

\begin{proof}[Proof of Theorem \ref{T:double-A-curve}]
Recall from Lemma \ref{L:pluricanonical-bases} that 
\[
\HH^0\bigl(C,\omega_C^m\bigr)
=\lspan \{ \omega_j \}_{j=m}^{mk}\oplus \lspan\{\eta_j \}_{j=m}^{mk}\oplus\lspan \{\chi_\ell\}_{\ell=-k(m-1)+1}^{k(m-1)-1}.
\]
Now, given a one-parameter subgroup $\rho$ as above, we will construct the requisite monomial basis $\B$ as a union
\[
\B=\B_{\omega} \cup \B_{\eta} \cup \B_{\chi},
\]
where $\B_{\omega}, \B_{\eta}$, and $\B_{\chi}$ are collections of degree 
$m$ monomials which map onto the bases of the subspaces spanned by 
$\{ \omega_j \}_{j=m}^{mk}, \{\eta_j \}_{j=m}^{mk}$, and $\{\chi_\ell\}_{\ell=-k(m-1)+1}^{k(m-1)-1}$, respectively.

To construct $\B_{\omega}$ and $\B_{\eta}$, we use Kempf's 
proof of the stability of Hilbert points of a rational normal curve. 
More precisely, consider the component $C_0$ of $C$ with the uniformizer $u_0$ at $0$.
Clearly, $\omega_{C}|_{C_0} \simeq \O_{\PP^1}(k-1)$. The restriction map 
$\HH^0\bigl(C, \omega_C\bigr) \rightarrow \HH^0 \bigl(\PP^1, \O_{\PP^1}(k-1) \bigr)$ identifies $\{x_i\}_{i=1}^{k}$ with a basis of 
$\HH^0 \bigl(\PP^1, \O_{\PP^1}(k-1)\bigr)$ given by $\{1,u_0,\dots, u_0^{k-1}\}$. 
Under this identification, the subspace $\lspan\{\omega_j\}_{j=m}^{mk}$ is identified with 
$\HH^0\bigl(\PP^1, \O_{\PP^1}(m(k-1))\bigr)$. 
Set $\lambda:=\sum_{i=1}^k \lambda_i/k$. Given a one-parameter subgroup $\widetilde{\rho}\co \gm \ra \SL(k)$ acting 
on $\{x_1,\dots, x_k\}$ 
diagonally with weights $(\lambda_1-\lambda, \dots, \lambda_k-\lambda)$, Kempf's result on the semistability of a rational normal 
curve in $\PP^{k-1}$ \cite[Corollary 5.3]{kempf},
implies the existence of a monomial basis $\B_\omega$ of $\HH^0\bigl(\PP^1, \O_{\PP^1}(m(k-1))\bigr)$ 
with non-positive $\widetilde{\rho}$-weight.  
Under the above identification, $\B_{\omega}$ is a monomial basis of $\lspan\{\omega_j\}_{j=m}^{mk}$ of $\rho$-weight
at most $m(mk-m+1)\lambda$. Similarly, if $\nu:=\sum_{i=1}^k\nu_i/k$, we deduce the existence 
of a monomial basis $\B_{\eta}$ of  $\lspan\{\eta_j\}_{j=m}^{mk}$ whose $\rho$-weight is at most $m(mk-m+1)\nu$. 
Since $\lambda+\nu=0$, it follows that the total $\rho$-weight of $\B_{\omega} \cup \B_{\eta}$ is non-positive.

Thus, to construct a monomial basis $\B$ of non-positive $\rho$-weight, 
it remains to construct a monomial basis $\B_{\chi}$ of non-positive $\rho$-weight
for the subspace 
\[
\lspan\{\chi_\ell\}_{\ell=-(m-1)k-1}^{(m-1)k-1} \subset \HH^0\bigl(C, \omega_C^m \bigr).
\]
In Lemma \ref{L:basis}, proved below, we show the existence of such a basis. Thus, we obtain the desired monomial basis $\B$ and finish the proof.
\end{proof}

Note that if we define the {\em weighted degree} by $\deg(x_i)=i$ and $\deg(y_i)=-i$, then a set $\B_{\chi}$ of $2k(m-1)-1$ degree $m$ monomials 
in $\{x_1,\ldots, x_k, $ $y_1, \ldots, y_k\}$ maps to a basis of $\lspan\{\chi_\ell\}_{\ell=-(m-1)k-1}^{(m-1)k-1} $ 
if and only if it satisfies the following two conditions:
\begin{enumerate}
\item Each monomial has both $x_i$ and $y_i$ terms,
\item Each weighted degree from $(m-1)k-1$ to $-(m-1)k+1$ occurs exactly once.
\end{enumerate}
We call such a set of monomials a {\em $\chi$-basis}. The following combinatorial lemma completes the proof of Theorem \ref{T:double-A-curve}.
\begin{lemma}\label{L:basis}
Suppose $\rho\co \gm\ra \SL(2k)$ is a one-parameter subgroup which 
acts on \\ $\{x_1, \ldots, x_{k}, y_1, \ldots, y_{k}\}$ diagonally with integer weights $\lambda_1, \ldots, \lambda_{k}, \nu_1, \ldots, \nu_{k}$
satisfying \\ $\sum_{i=1}^k (\lambda_i+\nu_i)=0$. Then there exists a $\chi$-basis with non-positive $\rho$-weight.
\end{lemma}

\begin{proof}[Proof of Lemma \ref{L:basis} for $m=2$]

Take the first $\chi$-basis to be
\begin{align*}
\B_1:=\left\{ 
\{x_{i}\,y_{k-i}\}_{1\leq i \leq k-1},  \{x_{i}\,y_{k-i+1}\}_{1 \leq i \leq k} \right\}.
\end{align*}
In this basis, all variables except $x_k$ and $y_k$ occur twice and $x_k$, $y_k$ occur once each. Thus 
\[
w_{\rho}(\B_1)=2(\lambda_1+\cdots+\lambda_{k-1})+2(\nu_1+\cdots+\nu_{k-1})+\lambda_k+\nu_k=-(\lambda_k+\nu_k).
\]
Take the second $\chi$-basis to be
\begin{align*}
\B_2:=\left\{ \{x_{k}\,y_{i}\}_{1\leq i \leq k},  \{x_{i}\,y_{k}\}_{1 \leq i \leq k-1} \right\}.
\end{align*}
We have
\[
w_{\rho}(\B_2)=k(\lambda_k+\nu_k)+\sum_{i=1}^{k-1} (\lambda_i+\nu_i)=(k-1)(\lambda_k+\nu_k).
\]
We conclude that for any one-parameter subgroup $\rho$, we have $(k-1)w_{\rho}(\B_1)+w_{\rho}(\B_2)=0$. It follows 
that either $\B_1$ or $\B_2$ gives a $\chi$-basis of non-positive weight.
\end{proof}

\begin{proof}[Proof of Lemma \ref{L:basis} for $m\geq 3$] 
We will prove the Lemma by exhibiting one collection of $\chi$-bases whose $\rho$-weights sum to a positive multiple of $\lambda_k +\nu_k$ and 
a collection of $\chi$-bases whose $\rho$-weights sum to a negative multiple of $\lambda_k+\nu_k$. Since, for any given one-parameter subgroup 
$\rho$, we have either $\lambda_k+\nu_k \geq 0$ or $\lambda_k +\nu_k \leq 0$, 
it follows at once that one of our $\chi$-bases must have non-positive weight. 

Throughout this section, we let $\iota$ be the involution exchanging $x_i$ and $y_i$.
We begin by writing down $\chi$-bases maximizing the occurrences of $x_k$ and $y_k$ while balancing the occurrences of the other variables.
Define $T_1$ as the set of all degree $m$ monomials having both $x_i$ and $y_i$ terms that belong to the ideal
\begin{align*}
(x_k, y_k)^{m-1}(x_1, \dots, x_{k}, y_1, \dots, y_{k}).
\end{align*}
The $\rho$-weight of $T_1$ is
\[
\left(k(m-1)+(2k-1)\binom{m-1}{2}\right)(\lambda_k+\nu_k) + (m-1)\sum_{i=1}^{k-1}(\lambda_i+\nu_i).
\]

Note that $T_1$ misses only the $m-2$ weighted degrees 
\[
k(m-3), k(m-5), \dots,-k(m-5),-k(m-3).
\]  
For each $s=1, \dots, k-1$, 
define a set of $m-2$ monomials having exactly these missing degrees by 
\begin{align*}
T_2(s) &: = \{x_{k}^{m-2-d}\,y^d_k\,(x_{k-s}x_{s})\ : \quad 1\leq d \leq m-2\} 
\end{align*}
For each $s$, the sets $T_1 \cup T_2(s)$ and $T_1 \cup \iota(T_2(s))$ are $\chi$-bases. Using $\sum_{i=1}^{k}(\lambda_i +\nu_i)=0$, one sees at once that the sum 
of the $\rho$-weights of such bases, as
$s$ ranges from $1$ to $k-1$, is a {\em positive multiple} of $(\lambda_k+\nu_k)$.

We now write down bases minimizing the occurrences of $x_k$ and $y_k$.  
We handle the case when $k$ is even and odd separately.

\subsubsection*{Case of even $k$:} If $k=2\ell$, we define the following set of monomials where the weighted degrees range from  $k(m-1)-1$ to $m$:
\begin{equation*}
S_1 := \left\{ 
\begin{aligned}
&x^{m-1-d}_{i}\,x^{d}_{i-1}\,y_{k+1-i}\ : &&\quad \ell+2 \leq i \leq k, \ 0 \leq d \leq m-1 \\
&x^{m-1-d}_{i}\,x^{d}_{i-1}\,y_{i-1}\ : &&\quad 2 \leq i \leq \ell+1, \ 0 \leq d \leq m-3
\end{aligned}\right\}
\end{equation*}
In the set $S_1 \cup \iota(S_1)$, the variables $x_k$ and $y_k$ occur $(m^2-m)-\binom{m}{2}$ times,  
$x_{\ell+1}$ and $y_{\ell+1}$ occur $(m^2-m)-1$ times, $x_{\ell}$ and $y_{\ell}$ occur $(m^2-m)-m$ times, 
and $x_1$ and $y_1$ occur $m^2-m-\left(\binom{m}{2}-1\right)$ times while all of the other variables occur $m^2-m$ times. To complete $S_1 \cup \iota(S_1)$ to a $\chi$-basis, we define, for each $s=1, \ldots, k-1$, the following set of monomials where the weighted degrees
range from $m-1$ to $1-m$:
\begin{equation*}
S_2(s):= \left\{
\begin{aligned}
&x_{\ell+1}\,y_{\ell}\,x_1^{m-2} \\
&x_{\ell}\,y_{\ell}\,(x_sy_s)^i\,x_1^{m-2i-2}\ :  &&\quad \text{for $0\leq 2i \leq m-2$}, \\
&x_{\ell}\,y_{\ell}\,(x_sy_s)^{i}\,y_1^{m-2i-2}\ :  &&\quad \text{for $0\leq 2i < m-2$},  \\
&(x_{k}\,y_{s}\,y_{k-s})\,(x_sy_s)^i\,x_1^{m-2i-3}\ :  &&\quad \text{for $0\leq 2i \leq m-3$}, \\
&(x_{k}\,y_{s}\,y_{k-s})\,(x_sy_s)^i\,y_1^{m-2i-3}\ : &&\quad \text{for $0\leq 2i < m-3$}, \\
&y_{\ell+1}\,x_\ell \,y_{1}^{m-2} && 
\end{aligned}
 \right\}
\end{equation*}
For each $s = 1, \ldots, k-1$, the sets $S_1 \cup \iota(S_1) \cup S_2(s)$ and $S_1 \cup \iota(S_1) \cup \iota(S_2(s))$ are 
$\chi$-bases.
We compute that in the union 
\begin{equation*}
\bigcup_{s=1}^{k-1} \bigl(S_1 \cup \iota(S_1) \cup S_2(s)\bigr) \cup \bigl(S_1 \cup \iota(S_1) \cup \iota\left(S_2(s)\right)\bigr)
\end{equation*}
of $2(k-1)$ $\chi$-bases the variables $x_k$ and $y_k$ each occurs 
\[
2(k-1)(m^2-m)-(k-1)(m^2-2m+2)
\]
times while all of the other variables occur 
\[
2(k-1)(m^2-m)+(m-2)(m-1)
\] 
times. 

Using the relation $\sum_{i=1}^{k}(\lambda_i +\nu_i)=0$, we conclude that the sum of the $\rho$-weights of all such $\chi$-bases is a {\em negative multiple} of 
$(\lambda_k + \nu_k)$.

\subsubsection*{Case of odd $k$:} If $k=2\ell+1$ is odd, $\chi$-bases whose $\rho$-weight is a negative multiple 
of $(\lambda_k + \nu_k)$ can be constructed analogously to the case when $k$ is even. 
For the reader's convenience, we spell out the details.  
We define the following set of monomials where the weighted degrees range from  $k(m-1)-1$ to $m-1$:
\begin{equation*}
S_1 :=  \left\{ 
\begin{aligned}
&x^{m-1-d}_{i}\,x^{d}_{i-1}\,y_{k+1-i}\ : &&\quad \ell+3 \leq i \leq k, \ 0 \leq d \leq m-1 \\
&x^{m-1-d}_{i}\,x^{d}_{i-1}\,y_{i-2}\ : &&\quad 3 \leq i \leq \ell+2, \ 0 \leq d \leq m-3 \\
&x_{\ell+2}\,y_{\ell}\,x_2^{m-2} && \\
&x_{\ell+1}\,y_{\ell}\,x_{2}^{m-2-d}\,x_1^d\ : &&\quad 0 \leq d \leq m-2 
\end{aligned}\right\}
\end{equation*}

In the set of monomials $S_1 \cup \iota(S_1)$, the variables $x_k$ and $y_k$ occur $\binom{m}{2}$ times, $x_{\ell+1}$ 
and $y_{\ell+1}$ occur $m^2-m-(m-1)$ times, 
and $x_1$ and $y_1$ occur $m^2-m - \binom{m-1}{2}$ times, while all of the other variables occur $m^2-m$ times.  Finally, for each $s=1, \ldots, k-1$, we define the following set of monomials where the weighted degrees range from $m-2$ to $2-m$:
\begin{equation*}
S_2(s):= \left\{
\begin{aligned}
&x_{\ell+1}\,y_{\ell+1}\,(x_sy_s)^i\,x_1^{m-2-2i}\ : &&\quad \text{for $0\leq 2i \leq m-2$}, \\
&x_{\ell+1}\,y_{\ell+1}\,(x_sy_s)^i\,y_1^{m-2-2i}\ : &&\quad \text{for $0\leq 2i < m-2$},  \\
&(x_k \,y_s \,y_{k-s})\,(x_sy_s)^i\,x_1^{m-3-2i}\ : &&\quad \text{for $0\leq 2i \leq m-3$}, \\
&(x_k \,y_s \,y_{k-s})\,(x_sy_s)^i\,y_1^{m-3-2i}\ : &&\quad \text{for $0\leq 2i < m-3$} 
\end{aligned}
\right\}
\end{equation*}
For each $s=1,\dots, k-1$, the sets 
$S_1 \cup \iota(S_1)\cup S_2(s)$ and $S_1 \cup \iota(S_1) \cup \iota(S_2(s))$ are $\chi$-bases.
We compute that in the union 
\begin{equation*}
\bigcup_{s=1}^{k-1} \bigl(S_1 \cup \iota(S_1) \cup S_2(s)\bigr) \cup \bigl(S_1 \cup \iota(S_1) \cup \iota(S_2(s)) \bigr)
\end{equation*}
of $2(k-1)$ $\chi$-bases the variables $x_k$ and $y_k$ each occurs 
$2(k-1)\binom{m}{2}+(k-1)(m-2)$
times while all of the other variables occur 
$2(k-1)(m^2-m)+(m-2)(m-1)$ times. 

Using the relation $\sum_{i=1}^{k}(\lambda_i +\nu_i)=0$, we conclude that
the total $\rho$-weight of these 
$\chi$-bases is a {\em negative multiple} of $(\lambda_k + \nu_k)$ and we are done.
\end{proof}

\subsection{Bicanonically embedded rosary}\label{S:monomial-bases-rosary}

We continue our study of the rosary $C$ defined in Section \ref{S:rosary}. In this section, we prove the 
Theorem \ref{T:main-trigonal} (2).

\begin{theorem}\label{T:rosary}
If $C \subset \PP \HH^0\bigl(C, \omega_{C}^2\bigr)$ is a bicanonically embedded rosary, then the Hilbert points $[C]_m$ 
are semistable for all $m \geq 2$.
\end{theorem}

\begin{corollary}[Theorem \ref{T:main-trigonal} (2)]\label{C:rosary}
Suppose $C \subset \PP \HH^0 \bigl(C, K_{C}^2 \bigr)$ is a generic bicanonically embedded smooth bielliptic curve of odd genus. 
Then the $m^{th}$ Hilbert point of $C$ is semistable for every $m \geq 2$.  
\end{corollary}

\begin{proof}[Proof of Corollary] 
This follows immediately from Theorem \ref{T:rosary} and Lemma \ref{L:rosary}.
\end{proof}

\begin{proof}[Proof of Theorem \ref{T:rosary}] 
We follow the notation of Lemma \ref{L:rosary-sections} (b). 
We need to show that for any one-parameter subgroup $\rho \co \GG_m \to \SL(3g-3)$ acting on the basis 
$\{x_i, y_i, z_i\,: \ i\in \ZZ_{g-1}\}$
of $\HH^0\bigl(C,\O_C(1) \bigr)=\HH^0\bigl(C,\omega_C^2 \bigr)$ 
diagonally, there is 
a monomial basis of $\HH^0 \bigl(C,\O_C(m) \bigr)=\HH^0 \bigl(C,\omega_C^{2m} \bigr)$ of non-positive $\rho$-weight.

We now define several monomial bases of $\HH^0 \bigl(C,\omega_C^{2m} \bigr)$. To begin, set
\begin{align*}
S_0 &:=\left\{ x_i^{m}, \quad x_i^{m-1}y_i: \qquad i\in \ZZ_{g-1}\right\}, \\ 
S_1 &:=\left\{
\begin{aligned}
&x_i^d\, z_i^{m-d}, &&x_i^d\, z_{i+1}^{m-d}: && i\in \ZZ_{g-1}, \quad 1\leq d \leq m-1 \\ 
&x_i^d\, y_iz_i^{m-d-1}, &&x_i^d\,y_i\,z_{i+1}^{m-d-1}: && i\in \ZZ_{g-1}, \quad 0\leq d \leq m-2 
\end{aligned}\right\}, \\
S_2 &:=\begin{cases} \left\{(y_{i-1}\,y_{i})^{\ell}\,z_i: \quad i\in \ZZ_{g-1} \right\} & \text{ if $m=2\ell+1$ is odd}, \\
\left\{(y_{i-1}\,y_{i})^{\ell}\,z^2_i: \quad i\in \ZZ_{g-1} \right\} & \text{ if $m=2\ell+2$ is even}.
\end{cases} \\
S'_2 &:=\begin{cases} \left\{(y_{i-1}\,y_{i})^{\ell}\,z_i: \quad i\in \ZZ_{g-1} \right\} & \text{ if $m=2\ell+1$ is odd}, \\
\left\{(y_{i-1}\,y_{i})^{\ell+1}: \quad i\in \ZZ_{g-1} \right\} & \text{ if $m=2\ell+2$ is even}.
\end{cases}
\end{align*}
Note that the choice of $S_0$ is prescribed by Lemma \ref{L:rosary-basis-bicanonical} (1--2) and 
that there are $(g-1)$ linearly independent
monomials of weight $j$ in $S_1$, for each $1\leq |j|\leq 2m-2$, and our choice of these monomials minimizes the occurrences
of $y_i$'s. Also, $S_2$ and $S'_2$ each contains $(g-1)$ linearly independent monomials of weight $0$. 
It follows that the following are monomial bases of $\HH^0\bigl(C,\omega_C^{2m} \bigr)$
\begin{align*}
\B_1^{+} &:=S_0\cup S_1\cup S_2, \\
\B_2^{+} &:=S_0\cup S_1\cup S'_2. 
\end{align*}

\begin{remark} When $g=3$ and $m$ is even, $S_2'$ contains only one element. In this case, 
we take $\B_1^{+}:=S_0\cup S_1\cup \{(y_{0}y_{1})^{\ell}z_0^2, (y_{0}y_{1})^{\ell+1}\}$ and 
$\B_2^{+}:=S_0\cup S_1\cup \{(y_{0}y_{1})^{\ell}z_1^2, (y_{0}y_{1})^{\ell+1}\}$.
\end{remark}

Let $X^{\rho}, Y^{\rho}, Z^{\rho}$ denote the sum of the $\rho$-weights of the $x_i$'s, $y_i$'s, $z_i$'s, respectively. 
In order to balance the occurrences of $x_i$'s and $z_i$'s, we consider the average of the $\rho$-weights of $\B_1^{+}$ and $\B_{2}^{+}$ 
and obtain 
\begin{equation*}
\frac{1}{2}\bigl(w_{\rho}(\B_1^{+})+w_{\rho}(\B_1^{+})\bigr)=(2m^2-2m+1)X^{\rho}+(3m-2)Y^{\rho}+(2m^2-2m+1)Z^{\rho}.
\end{equation*}

Next we define an alternate pair of monomial bases maximizing the occurrences of $y_i$'s. To do so, we set
\begin{equation*}
T_1 :=\left\{
\begin{aligned}
&x_i^{d} \,y_i^{m-d}, &&x_i^{d+1}\,y_i^{m-d-2}\,z_i && :i\in \ZZ_{g-1}, \quad 0\leq d\leq m-2 \\
&y_i^d \,z_i^{m-d}, &&y_i^d \,z_{i+1}^{m-d} && :i\in \ZZ_{g-1}, \quad 2\leq d \leq m-1 \\
&y_i \,z_i^{m-1}, &&y_i \,z_{i+1}^{m-1} && :i\in \ZZ_{g-1}.  
\end{aligned}\right\} 
\end{equation*}
and define 
\begin{align*}
\B_1^{-} &:=S_0\cup T_1\cup S_2, \\
\B_2^{-} &:=S_0\cup T_1\cup S'_2. 
\end{align*}
One easily checks that $\B_1^{-}$ and $\B_2^{-}$ are monomial bases of $\HH^0\bigl(C,\omega_C^{2m}\bigr)$ 
and that the average of their $\rho$-weights is 
\begin{equation*}
m^2X^{\rho}+(2m^2-m)Y^{\rho}+m^2Z^{\rho}.
\end{equation*}

Using $X^{\rho}+Y^{\rho}+Z^{\rho}=0$, we obtain 
\[
(m^2-m)\bigl(w_{\rho}(\B_1^{+})+w_{\rho}(\B_1^{+})\bigr)+(2m^2-5m+3)\bigl(w_{\rho}(\B_1^{-})+w_{\rho}(\B_2^{-})\bigr)=0
\]
for any one-parameter subgroup $\rho$. Lemma \ref{L:covering} 
now finishes the proof of the theorem.
\end{proof}

\section{Non-semistability results}
\label{S:non-semistability}

\subsection{Canonically embedded rosary}\label{S:canonical-rosary}
Let $C$ denote the rosary defined in Section \ref{S:rosary}. In this section, we analyze finite Hilbert 
stability of the {\em canonical embedding} of $C$. We find that $C$ is the first known example of a canonical curve in arbitrary (odd) genus 
such that stability of its Hilbert points depends on $m$: $[C]_m$ is  
semistable for large $m$ but becomes non-semistable for small $m$. More precisely, we have the following result. 
\begin{theorem}\label{T:rosary-canonical}
Let $C \subset \PP \HH^0\bigl(C, \omega_{C}\bigr)$ be the canonically embedded rosary of odd genus $g\geq 5$.
Then $[C]_m$ is semistable if and only if $g\leq 2m+1$.
\end{theorem}
\begin{proof} We follow the notation of Lemma \ref{L:rosary-sections} (a). 
First, we show that $[C]_m$ is semistable for $g\leq 2m+1$. This is accomplished by the same technique as in the 
previous sections, namely by using Lemma \ref{L:rosary-basis-canonical} to find non-positive monomial bases of $\HH^0\bigl(C,\omega_C^m\bigr)$.  
Let $\rho\co \GG_m \to \SL(g)$ be a one-parameter subgroup acting on the basis $(\omega_0,\dots, \omega_{g-2},\eta)$ 
diagonally with weights $(\rho_0,\dots, \rho_{g-2}, \rho_{g-1})$.
Set $W:=\sum_{i=0}^{g-2} \rho_i=-\rho_{g-1}$. We will construct bases in which all the $\omega_i$
appear equally often and hence these bases have $\rho$-weights that are multiples of $W$:

First, we find a basis in which $\eta$ appears as seldom as possible. We define a 
basis $\B^{+}$ to be 
the following set of monomials:
\begin{equation*} \B^{+}:=\left\{
\begin{aligned}
&\omega_i^m, \ \omega_i^{m-1}\eta  &&:i\in \ZZ_{g-1}, \\
&\omega_i^{m-d}\omega_{i-1}^{d}, \ \omega_i^d\omega_{i-1}^{m-d} && :i\in \ZZ_{g-1}, \quad 1\leq m-2d\leq m-2, \\
&\omega_i^{m-d-1}\omega_{i-1}^{d} \eta, \ \omega_i^d\omega_{i-1}^{m-d-1} \eta && :i\in \ZZ_{g-1}, \quad 2 \leq m-2d \leq m-2, \\
&\begin{cases} \omega_i^\ell \omega_{i-1}^\ell & \text{ if $m=2\ell$} \\
\omega_i^{\ell-1} \omega_{i-1}^{\ell-1} \eta & \text{ if $m=2\ell-1$} \end{cases} && : i\in \ZZ_{g-1}.
\end{aligned} \right\}
\end{equation*}

The $\rho$-weight of $\B^{+}$ is
\[
(2m^2-2m+1)W + (m-1)(g-1)\rho_{g-1}=\bigl(2m^2-2m+1-(m-1)(g-1)\bigr)W.
\]

We now find a basis in which $\eta$ appears as often as possible. Namely, we set 
\begin{equation*} 
\B^{-}:=\left\{
\begin{aligned}
&\omega_i^m, &&\omega_i^{m-1}\eta  &&:i\in \ZZ_{g-1}, \\
&\omega_i^d \eta^{m-d}, &&\omega_i \omega_{i-1}^{d+1}\eta^{m-d-2} &&:i\in \ZZ_{g-1}, \quad 1\leq d\leq m-2, \\
&\omega_i \omega_{i-1} \eta^{m-2} && &&:i\in \ZZ_{g-1}.
\end{aligned} \right\}
\end{equation*}
Then the $\rho$-weight of the basis $\B^{-}$ is
\[
(m^2+m-1)W+(m-1)^2(g-1)\rho_{g-1}=\bigl(m^2+m-1-(m-1)^2(g-1)\bigr)W.
\]
If $(g,m)\neq (5,2)$ and $g\leq 2m+1$, then either $\B^{+}$ or $\B^{-}$ has non-positive weight with respect to $\rho$. 
If $(g,m)=(5,2)$, then it is easy to find three explicit monomial bases that accomplish the same result.
This finishes the proof of semistability. 

Conversely, suppose $g\geq  2m+3$. Consider 
the one-parameter subgroup $\rho$ acting with weight $(-1)$ on $\omega_i$'s and weight $g-1$ on $\eta$.
If $\B$ is a monomial basis of $\HH^0 \bigl(C,\O_C(m) \bigr)=\HH^0 \bigl(C,\omega_C^m \bigr)$, 
then for each odd $\ell$ each monomial of weight $\pm(m-\ell)$ with respect to $\GG_m\subset \Aut(C)$ 
necessarily has an $\eta$ term (see Lemma \ref{L:rosary-basis-canonical}). 
It follows that the variable $\eta$ of weight $(g-1)$ occurs at least $(m-1)(g-1)$ times among  
monomials of $\B$. The remaining at most $m(2m-1)(g-1)-(m-1)(g-1)$ variables occurring in $\B$ all have
weight $(-1)$. It follows that the total $\rho$-weight of 
$\B$ is at least 
\begin{multline*}
(g-1)(m-1)(g-1)-\bigl(m(2m-1)(g-1)-(m-1)(g-1)\bigr) \\=(g-1)\left((m-1)(g-1)-(2m^2-2m+1)\right) \geq (g-1)(2m-3)> 0.
\end{multline*}
Thus $\rho$ destabilizes $C$.
\end{proof}

\subsection{Canonically embedded bielliptic curves}
Our main result raises a natural question of whether Hilbert points of 
smooth canonically embedded curves can at all be non-semistable. An indirect way to see that the answer is affirmative is as follows. 
By \cite[Section 5]{hassett-hyeon_flip}, there is an open locus in $\bigl(\overline{H}_{g,1}^{\, m}\bigr)^{ss}$ over whose $\SL(g)$-quotient, 
the tautological GIT polarization is a positive multiple of $s^m_g\lambda-\delta$, where $\lambda$ and $\delta$ are the Hodge and boundary classes
and 
\begin{equation}\label{E:polarization}
s_g^m:=8+\frac{4}{g}-\frac{2(g-1)}{gm}+\frac{2}{gm(m-1)}.
\end{equation}
By generalizing the proof of \cite[Proposition 4.3]{CH}, 
we see that if $B\ra \Mg{g}$ is a family of stable curves whose generic fiber is canonically embedded and the slope $(\delta\cdot B)/(\lambda\cdot B)$
is greater than $s_g^m$, 
then every curve in $B$ with a well-defined $m^{th}$ Hilbert point
must have non-semistable $m^{th}$ Hilbert point.

Two observations now lead to a candidate for a non-semistable canonically embedded {\em smooth} curve. 
The first is that
$s_g^m\leq 8$ for $g\geq 2m+1+1/(m-1)$. The second is that families 
of bielliptic curves of slope $8$ can be constructed by taking a double cover of a constant family of elliptic curves
(e.g. \cite{xiao, barja}). 
In the following result, we establish that canonical bielliptic curves indeed become non-semistable 
for small values of $m$, and show that a generic canonical bielliptic curve is semistable for $m$ large enough.

\begin{theorem}\label{T:bielliptic-change}
A canonically embedded smooth bielliptic curve of genus $g$ has non-semistable
$m^{th}$ Hilbert point for all $m \leq (g-3)/2$. A generic canonically embedded bielliptic curve of odd genus has 
semistable $m^{th}$ Hilbert point for all $m\geq (g-1)/2$.
\end{theorem}
\begin{proof}
Let $C$ be a bielliptic canonical curve. Then $C$ is a quadric section of a projective cone over an elliptic curve 
$E\subset \PP^{g-2}$ 
embedded by a complete linear system of degree $g-1$. Choose projective coordinates $[x_0:\ldots: x_{g-1}]$
such that the vertex of the cone has coordinates $[0:0:\ldots:0:1]$.
Let $\rho$ be the one-parameter subgroup of $\SL(g)$ acting with weights $(-1,-1,\dots, -1, g-1)$.
For every monomial basis of $\HH^0\bigl(C, \O_C(m)\bigr)$, the number of monomials of $\rho$-weight $-m$,
that is degree $m$ monomials in the variables $x_0,\dots, x_{g-2}$, 
is bounded above by $h^0\bigl(E,\O_E(m)\bigr)=m(g-1)$.
The remaining at least $(m-1)(g-1)$ elements of the monomial basis 
have $\rho$-weight at least $g-m$. Thus the $\rho$-weight of any monomial basis of 
$\HH^0 \bigl(C,\O_C(m) \bigr)$ is at least 
\begin{equation}\label{E:bielliptic-weight}
(m-1)(g-1)(g-m)-m^2(g-1)
=(g-1)\left(m(g+1)-2m^2-g\right).
\end{equation}
If $m\leq (g-3)/2$, then \eqref{E:bielliptic-weight} is positive, and thus $\rho$ destabilizes $[C]_m$.

To prove the generic semistability of bielliptic curves in the range $m\geq (g-1)/2$, 
note that we have already seen that the canonically embedded rosary of odd genus $g\geq 5$ has semistable $m^{th}$ Hilbert point
if and only if $g\leq 2m+1$ (Theorem \ref{T:rosary-canonical}).
It remains to observe that the rosary deforms flatly to a smooth bielliptic curve. 
This is accomplished in Lemma \ref{L:rosary} below. 
\end{proof}

\begin{lemma}
\label{L:rosary}
The rosary of genus $g\geq 4$ deforms flatly to a smooth bielliptic curve.
\end{lemma}
\begin{proof} Let $C$ be the rosary considered in Section \ref{S:rosary}. 
Consider $\PP^{g-2}$ with projective coordinates $[x_0:\ldots:x_{g-2}]$ 
and define $E\subset\PP^{g-2}$ to be the union of $g-1$ lines $L_i : \{x_{i+1}=\cdots=x_{i+g-3}=0\}$
for $i\in \ZZ_{g-1}$. Then $E$ is a nodal curve of arithmetic genus $1$. Since $\HH^1\bigl(E,\O_E(1)\bigr)=0$, we can deform $E$ flatly inside $\PP^{g-2}$ 
to a smooth elliptic curve by \cite[p.83]{Kol}. Using the basis $(\omega_0, \dots, \omega_{g-2}, \eta)$ 
of $\HH^0\bigl(C,\omega_C\bigr)$ described in Lemma \ref{L:rosary-sections} (a),
we observe that the canonical embedding of $C$ is cut out by the quadric 
\[
x_0x_1+x_1x_2+\cdots+x_{g-2}x_0=x_{g-1}^2
\]
on the projective cone over $E$ in $\PP^{g-1}$.  Since $E$ deforms to a smooth elliptic curve, 
it follows that $C$ deforms to a smooth bielliptic curve.
\end{proof}

\begin{remark}[Trigonal curves of higher Maroni invariant] Theorem \ref{T:main-trigonal} (1) shows that 
a generic trigonal curve with Maroni invariant $0$ has semistable $m^{th}$ Hilbert point for all $m\geq 2$. 
In joint work of the second author with Jensen, it is shown that every trigonal curve with Maroni invariant $0$ 
has semistable $2^{nd}$ Hilbert point and every trigonal curve with a positive Maroni invariant has non-semistable $2^{nd}$ Hilbert point 
\cite{fedorchuk-jensen}. In view of the asymptotic stability of canonically embedded 
curves (Theorem \ref{T:asymptotic-stability}), 
this result suggests that {\em every} smooth trigonal curve of Maroni invariant $0$ has semistable $m^{th}$ Hilbert
point for every $m\geq 2$. One also expects that for a generic smooth trigonal curve of positive Maroni invariant 
already the third Hilbert point is semistable. Indeed, Equation \ref{E:polarization} shows that the polarization on an open subset of
$\bigl(\overline{H}_{g,1}^{\, 3}\bigr)^{ss}\gitq \SL(g)$ is a multiple of 
\begin{equation*}
\left(\frac{22}{3}+\frac{5}{g}\right)\lambda-\delta.
\end{equation*}
On the other hand, the maximal possible slope for a family of generically smooth trigonal curves of genus $g$
 is $36(g+1)/(5g+1)$ by \cite{stankova}. 
We note that 
\[
36(g+1)/(5g+1)\leq \left(\frac{22}{3}+\frac{5}{g}\right)
\]
whenever $(g-3)(2g-5)\geq 0$. Thus we expect that the $3^{rd}$ Hilbert 
point of every canonically embedded smooth trigonal curve of genus $g\geq 4$ is stable. 
\end{remark}

\section{Stability of bicanonical curves}\label{S:wiman}
While the major theme of this paper is establishing finite Hilbert semistability of very singular
curves, our methods can be used to establish stability of smooth curves as well. In fact, 
the original motivation for our work is the problem of stability of low degree Hilbert
points of smooth bicanonical curves. 
\begin{conjecture}[I. Morrison]\label{conjecture} 
A smooth bicanonical curve of genus $g\geq 3$ has stable $m^{th}$ Hilbert point whenever
$(g,m)\neq (3,2)$.
\end{conjecture}
This problem was implicitly stated by Morrison 
\cite{morrison-git} in the wider context of GIT approaches to the log minimal model program for $\M_g$. 
In fact, it follows from the conjectural description, due to Hassett and Hyeon, of the second flip of $\M_g$
as the GIT quotient of the variety of $6^{th}$ Hilbert points of bicanonical genus $g$ curves 
that almost all bicanonically embedded Deligne-Mumford stable curves should have stable $m^{th}$ Hilbert points 
for every $m\geq 6$ \cite[Section 7.5]{morrison-git}. 

Here, we make a step toward Conjecture \ref{conjecture} by establishing the following result. 
\begin{theorem}[Stability of generic bicanonical curves]\label{T:bicanonical}
A generic bicanonically embedded smooth curve of genus $g\geq 3$ has stable $m^{th}$ Hilbert
point for every $m\geq 3$. In addition, a generic bicanonically embedded smooth curve of genus $g\geq 4$ 
has semistable $2^{nd}$ Hilbert point.
\end{theorem}

Our proof of Theorem \ref{T:bicanonical} begins with the original idea of Morrison and Swinarski 
\cite{morrison-swinarski} in that we also consider the {\em Wiman hyperelliptic curves} 
and apply Kempf's instability results \cite{kempf}.
Our strategy is however different in that instead of using symbolic computations 
with the ideal of the Wiman curve as in \cite{morrison-swinarski},
we exploit the high degree of symmetry of the Wiman curve, 
together with the fact that it is defined by a single equation,
to construct monomial bases by hand. We establish stability of the Wiman curve in Theorem \ref{T:wiman}, which immediately implies
Theorem \ref{T:bicanonical} by openness of semistability.

\subsection{Wiman curves}\label{S:wiman-curves}
Recall that a genus $g$ curve $C$ is a {\em Wiman curve} if it is defined by the equation
\begin{equation}\label{E:wiman}
w^2=z^{2g+1}+1. 
\end{equation} 
By \cite[Section 6]{morrison-swinarski}, we have 
\begin{align}\label{E:wiman-sections}
\HH^0\bigl(C, K_C^2\bigr)
=\CC\left\langle z^i\frac{(dz)^2}{w^2}\right\rangle_{0\leq i \leq 2g-2} \bigoplus
\CC\left\langle z^jw\frac{(dz)^2}{w^2}\right\rangle_{0\leq j \leq g-3}.
\end{align}

Since $C$ is a smooth curve, $\vert K_C^2 \vert$ defines a closed embedding 
$C\hookrightarrow \PP^{3g-4}$ for $g\geq 3$. From now on, we let $\O_C(1)=K_C^2$.
When discussing global sections of $\HH^0\bigl(C, \O_C(m)\bigr)=\HH^0\bigl(C,K_C^{2m}\bigr)$, we simply 
write $f(z,w)$ to denote an element $f(z,w)(dz)^{2m}/w^{2m}$.
We also fix once and for all a distinguished basis of $\HH^0\bigl(C,\O_C(1)\bigr)$ given by the following 
$3g-3$ functions: 
\begin{align*}
x_i&:=z^{i}, \quad 0\leq i \leq 2g-2, \\
y_j&:=z^{j}w, \quad 0\leq j \leq g-3.
\end{align*}
For $m\geq 1$ and $k\leq m$, a
monomial of the form $\prod_{a=1}^k x_{i_a} \prod_{b=1}^{m-k} y_{j_b}$ will be called a {\em $(k,m-k)$-monomial}.

The space of $(k,m-k)$-monomials in $\Sym^m \HH^0\bigl(C,\O_C(1)\bigr)$ maps injectively into 
$\HH^0\bigl(C, \O_C(m)\bigr)$ and we denote its image by $W(k,m-k)$. 
We note that  
\begin{equation*}
W(k, m-k) = \bigoplus_{d=0}^{(2g-2)k+(g-3)(m-k)} \CC\left\langle z^d w^{m-k}\right\rangle. 
\end{equation*}

For every $k\leq m-2$, Equation \eqref{E:wiman} gives rise to an injective linear map 
\[
r\colon W(k,m-k) \ra W(k+2, m-k-2),
\] defined by 
$r(z^d w^{m-k})=z^d(z^{2g+1}+1)w^{m-k-2}$, that realizes $W(k,m-k)$ as the subspace of $W(k+2,m-k-2)$. 
We record that 
\begin{equation*}
\dim_{\CC} W(k+2,m-k-2)/ r\bigl(W(k,m-k)\bigr) =2g+2, 
\end{equation*} 
and that there are isomorphisms 
\begin{align}
\HH^0\bigl(C,\O_C(m)\bigr) &\simeq W(m, 0)\oplus W(m-1,1), \label{E:isomorphism}\\ 
\HH^0\bigl(C,\O_C(m)\bigr) &\simeq \bigoplus_{k=0}^{m} W(k, m-k)/r\bigl(W(k-2,m-k+2)\bigr).  \label{E:isomorphism-2}
\end{align}

\begin{definition}\label{D:mk-basis}
If $V\subset \HH^0\bigl(C,\O_C(m)\bigr)$ is a linear subspace, 
a monomial basis of $V$ composed of $(k,m-k)$-monomials is called a {\em $(k,m-k)$-monomial basis}.
\end{definition}

\begin{lemma}\label{L:wiman-hilbert-point}
The $m^{th}$ Hilbert point of $C\hookrightarrow \PP\HH^0\bigl(C, K_C^2\bigr)$ is well-defined.
\end{lemma}
\begin{proof}
We need to show that $\Sym^m \HH^0 \bigl(C,\O_C(1) \bigr) \ra \HH^0 \bigl(C,\O_C(m) \bigr)$ is surjective. 
This follows immediately from the identification $\HH^0 \bigl(C,\O_C(m) \bigr)\simeq W(m, 0)\oplus W(m-1,1)$. 
\end{proof}
We recall that $\Aut(C)\simeq \mu_{4g+2}$, the cyclic group of order $4g+2$ \cite{wiman}. The action of the generator is given by 
\[
\zeta \cdot z=\zeta^{2} z, \qquad
\zeta \cdot w=\zeta^{2g+1}w.
\]
We immediately obtain the following observation.
\begin{lemma}\label{L:wiman-mult}
$\HH^0 \bigl(C,K_C^2\bigr)$ is a multiplicity-free representation of $\Aut(C)\simeq \mu_{4g+2}$ 
and the basis $\{x_0,\dots, x_{2g-2}, y_0, \dots, y_{g-3}\}$
is compatible with the irreducible decomposition of $\HH^0 \bigl(C,K_C^2\bigr)$.
\end{lemma}
\begin{proof}
Consulting Equation \eqref{E:wiman-sections}, we see that the weights of the $\mu_{4g+2}$-action
on the listed generators are $2i-4g+2$, where $0\leq i \leq 2g-2$, and $2j-2g+3$, where $0\leq j \leq g-3$.
\end{proof}

\begin{theorem}\label{T:wiman}
The bicanonically embedded Wiman curve $C\subset \PP\HH^0 \bigl(C,K_C^2\bigr)$
has stable $m^{th}$ Hilbert point for every $m\geq 3$ if $g\geq 3$, 
and has semistable $2^{nd}$ Hilbert point if $g\geq 4$.  
\end{theorem}

\begin{proof}[Proof of Theorem \ref{T:wiman}] 
Lemma \ref{L:wiman-mult} and Proposition \ref{P:kempf}
imply that it suffices to check stability of $C$ with respect to 
one-parameter subgroups acting diagonally on the distinguished basis
$\{x_0, \dots, x_{2g-2}, y_0, \dots, y_{g-3}\}$ of $\HH^0\bigl(C,\O_C(1) \bigr)=\HH^0\bigl(C,K_C^2 \bigr)$.
Suppose $\rho\colon \mathbb{G}_m \ra \SL(3g-3)$ is a one-parameter subgroup acting diagonally on this basis. 
We need to show that there is a monomial basis of $\HH^0 \bigl(C,\O_C(m) \bigr)$ 
whose $\rho$-weight is negative if $m\geq 3$ (resp., non-positive if $m=2$).
We do this in Corollary \ref{C:wiman-2nd} for $m=2$ 
and Corollary \ref{T:wiman-m-3} for $m\geq 3$.
\end{proof}
Let $\{\lambda_i\}_{i=0}^{2g-2}$ be the weights with which $\rho$ acts on $\{x_i\}_{i=0}^{2g-2}$ and let
$\{\nu_j\}_{j=0}^{g-3}$ be the weights with which $\rho$ acts on $\{y_j\}_{j=0}^{g-3}$. 
We also set $\Lambda:=\sum_{i=0}^{2g-2} \lambda_i$ and $N:=\sum_{j=0}^{g-3} \nu_j$.
Note that $\Lambda+N=0$.

Before proceeding to the construction of the requisite monomial bases, we introduce additional terminology.
A multiset $\mathbb{S}=\{\B_1, \dots, \B_s\}$ of (monomial) bases 
of a subspace $V\subset \HH^0\bigl(C,\O_C(m)\bigr)$ will be called
a {\em (monomial) multibasis} of $V$. 
If $\mathbb{S}=\{\B_k\}_{k=1}^s$ and $\mathbb{T}=\{\mathcal{R}_\ell\}_{\ell=1}^{t}$,
we will write $\mathbb{S}\cup \mathbb{T}$ to denote their concatenation. We will simply write $d\cdot \mathbb{S}$ to 
denote $\cup_{r=1}^{d} \mathbb{S}$. 

If $\rho$ is a one-parameter subgroup of $\SL(3g-3)$, 
we define the $\rho$-weight of $\mathbb{S}=\{\B_k\}_{k=1}^s$ to be 
\[
w_{\rho}(\mathbb{S}):=\frac{1}{s}\sum_{k=1}^s w_{\rho}(\B_k).
\]
Our motivation for considering
multibases comes from an elementary observation that existence of a monomial multibasis of non-positive (negative)
$\rho$-weight implies existence of a monomial basis of non-positive (negative)
$\rho$-weight. Multibases have the following useful property:
If $\mathbb{S}_1=\{\B_k\}_{k=1}^{s}$ and $\mathbb{S}_2=\{\mathcal{R}_\ell\}_{\ell=1}^{t}$ 
are multibases of subspaces $V_1, V_2\subset \HH^0\bigl(C,\O_C(m)\bigr)$ and
$V_1 \cap V_2=\{0\}$, then we can form the multibasis 
\[
\mathbb{S}_1+\mathbb{S}_2:=\left\{\B_k\cup \mathcal{R}_\ell \right\}_{1\leq k \leq s, \ 1\leq \ell \leq t}
\]
of $V_1+V_2$.
Evidently, $w_{\rho}(\mathbb{S}_1+\mathbb{S}_2)=w_{\rho}(\mathbb{S}_1)+w_\rho(\mathbb{S}_2)$.

We say that a monomial 
multibasis $\mathbb{S}$ is $X$-balanced if the variables
$\{x_i\}_{i=0}^{2g-2}$ occur the same number of times in $\mathbb{S}$. Similarly, we define $Y$-balanced monomial
multibases. Finally, $\mathbb{S}$ will be called {\em balanced} if it is both $X$- and $Y$-balanced.
The $\rho$-weight of a balanced monomial multibasis is a linear combination of $\Lambda$ and $N$.

\subsection{Key combinatorial lemmas}
\begin{lemma}\label{L:path-lemma}
Suppose $x_0, \dots, x_n$, $y_0, \dots, y_m$ are weighted variables such that  
$\deg x_i =i$ for $0\leq i \leq n$, and $\deg y_j=j$ for $0\leq j \leq m$.
Then there exists a multiset of quadratic monomials $S=\{x_i y_j\}_{(i,j)\in I}$ satisfying the following conditions:
\begin{enumerate}
\item Every degree in the range $[0, n+m]$ occurs $|S|/(n+m+1)$ times in $S$.
\item Each variable $x_i$ occurs $|S|/(n+1)$ times in $S$. 
\item Each variable $y_j$ occurs $|S|/(m+1)$ times in $S$. 
\end{enumerate}
\end{lemma}

\begin{proof}
Let $c_{ij}=\binom{i+j}{i}\binom{n+m-i-j}{n-i}$. Then $c_{ij}$'s
satisfy the following:
\begin{enumerate}
\item[(i)] $\sum\limits_{i+j=d} c_{ij}$ is the same for all $d$ in the range $[0, n+m]$.
\item[(ii)] $\sum \limits_{j=0}^m c_{ij}$  is the same for all $0\leq i\leq n$.
\item[(iii)] $\sum \limits_{i=0}^n c_{ij}$  is the same for all $0\leq j\leq m$.
\end{enumerate}
The multiset $S$ in which the monomial $x_iy_j$ occurs $c_{ij}$ times satisfies all requisite conditions. 
\end{proof}

Using preceding lemmas, 
we prove several results that enable our proof of Theorem \ref{T:wiman}. 
\begin{prop}\label{P:balanced-rational-curve}
Let $x_i:=z^i$ for $0\leq i\leq n$. 
For every $0\leq k\leq n$, there exists an $X$-balanced quadratic monomial 
multibasis $\mathbb{H}^{n}_{k}$
of $Z_k:=\lspan\{z^{i} \ : \ k\leq i \leq 2n-k\}$.
\end{prop}
\begin{proof}
To keep track of the number of appearances of variables $x_i$'s in multibases, 
we assume that a one-parameter subgroup $\rho\co \gm \ra \operatorname{GL}(n+1)$ acts 
on $\{x_i\}_{i=0}^{n}$ with weights $\{\lambda_i\}_{i=0}^{n}$. If $\mathbb{S}$ is a fixed multibasis 
of $Z_k$, then $w_{\rho}(\mathbb{S})$ is a linear function in $\lambda_i$'s. Denote $\Lambda:=\sum_{i=0}^n \lambda_i$.
Evidently, $\mathbb{S}$ is $X$-balanced if and only if 
$w_{\rho}(\mathbb{S})=\frac{2(2n-2k+1)}{n+1}\Lambda$ for every $\rho$.

We proceed by descending induction on $k$.
If $k=n$, then $\mathbb{H}^n_{n}:=\{x_ix_{n-i}\}_{i=0}^{n}$ 
is an $X$-balanced quadratic multibasis of $Z_n=\CC\langle z^n\rangle$. 

Suppose now $k\leq n-1$. Consider the following monomial bases of $Z_k$:
\begin{align*}
\B^{-}&:=\{x_i\, x_{k+i}, \ x_{i+1}\, x_{k+i} \ : \ 0\leq i \leq n-k-1\}\cup\{x_{n-k}\, x_{n}\}, \\
\B^{+}&:=\{x_0\, x_i \ : k\leq i \leq n\} \cup \{x_n\, x_i \ : \ 1\leq i \leq n-k\}.
\end{align*}
Their weights are 
\begin{align*}
w_{\rho}(\B^{-}) =\lambda_0+2\sum_{i=1}^{n-k}\lambda_i+2\sum_{i=k}^{n-1}\lambda_i+\lambda_n,\qquad
w_{\rho}(\B^{+}) =(n-k)(\lambda_0+\lambda_n)+\sum_{i=k}^{n} \lambda_i +\sum_{i=0}^{n-k} \lambda_i.
\end{align*}
If $k=0$, then $\mathbb{H}^{n}_0:=n \cdot \B^{-} \cup \B^{+}$ is an $X$-balanced monomial basis of $Z_0$.
If $k\geq 1$, then let 
 $\mathbb{H}_{k+1}^{n}$ be a balanced monomial multibasis of $Z_{k+1}$, which exists by the induction assumption. 
Let
$\mathbb{T}_0:=\mathbb{H}_{k+1}^n+\{x_ix_{k-i}\ \colon \ 1\leq i\leq k-1\}+\{x_{n-i}x_{n-k+i}\ \colon \ 1\leq i\leq k-1\}$
be a multibasis of $Z_{k}=Z_{k+1}+\CC\langle z^k\rangle+\CC\langle z^{n-k}\rangle$. Then
\begin{equation*}
w_{\rho}(\mathbb{T}_0)
=2\frac{(2n-2k-1)}{n+1}\Lambda+\frac{2}{k-1}\left(\sum_{i=1}^{k-1} \lambda_i+\sum_{i=n-k+1}^{n-1}\lambda_{i}\right).
\end{equation*}
It follows that the weight of $\mathbb{T}^{-}:=(k-1)\cdot \mathbb{T}_0\cup \B^{-}$ is
\[
\frac{1}{k}\left(\frac{2(k-1)(2n-2k-1)}{(n+1)}\Lambda+4\Lambda-3(\lambda_0+\lambda_n)\right)
\]
and the weight of $\mathbb{T}^{+}:=(k-1)\cdot \mathbb{T}_0\cup 2\cdot \B^{+}$ is
\[
\frac{1}{k+1}\left(\frac{2(k-1)(2n-2k-1)}{(n+1)}\Lambda+4\Lambda+(2n-2k-2)(\lambda_0+\lambda_n)\right).
\]
It follows that the multibasis $\mathbb{H}_k^{n}:=k(2n-2k-2) \cdot \mathbb{T}^{-} \cup 3(k+1)\cdot \mathbb{T}^{+}$ is 
a well-defined $X$-balanced monomial multibasis of $Z_k$. 
\end{proof}

\begin{remark}
The statement of Proposition \ref{P:balanced-rational-curve} for $k=0$ is equivalent to semistability of the $2^{nd}$ 
Hilbert point of a rational normal curve of degree $n$,
proved by Kempf in \cite[Corollary 5.3]{kempf}. 
A geometric interpretation of the remaining cases is more elusive. 
\end{remark}
\begin{prop}\label{P:V-bases}
There exists a balanced $(k,m-k)$-monomial multibasis $\mathbb{S}(k,m-k)$ of the space $W(k,m-k)\subset \HH^0\bigl(C,\O_C(m)\bigr)$.
\end{prop}
\begin{proof} We proceed by induction on $m$. The base case is $m=1$. Here, we can even find a 
balanced basis: If $k=0$, then $\{y_0,\dots, y_{g-3}\}$ is a balanced $(0,1)$-monomial basis of $W(0,1)$; 
if $k=1$, then $\{x_0,\dots, x_{2g-3}\}$ is a balanced $(1,0)$-monomial basis of $W(1,0)$. 

Suppose now that $m\geq 2$ and $k\geq 1$. Then $\mathbb{S}(k-1,m-k)$ exists by the induction assumption.
Write $\mathbb{S}(k-1,m-k)=\{\B_\ell\}_{\ell=1}^{r}$, where each
$\B_\ell=\{e^{\ell}_d\}_{d=0}^{(k-1)(2g-2)+(m-k)(g-3)}$ is a $(k-1,m-k)$-monomial basis of $W(k-1,m-k)$,
and where we choose the indexing so that the monomial $e^{\ell}_d$ maps to $z^d w^{m-k}$ 
in $\HH^0\bigl(C,\O_C(m)\bigr)$. Next, let $\deg(e^{\ell}_d)=d$ and $\deg(x_i)=i$, so that the degree 
corresponds to the power of $z$ occurring in a monomial.
Consider the multiset $S_{\ell}=\{x_i e^{\ell}_d\}_{(i,d)\in I}$
satisfying Lemma \ref{L:path-lemma}:
\begin{enumerate}
\item If we write $x_i e^{\ell}_d=z^{d+i} w^{m-k}$, then each power of $z$ occurs the same number of times.
\item Each index $0\leq i\leq 2g-2$ occurs the same number of times in $S_{\ell}$.
\item Each index $0\leq d\leq (k-1)(2g-2)+(m-k)(g-3)$ occurs the same number of times in $S_{\ell}$.
\end{enumerate}
Condition (1) implies that we can arrange the elements of $S_{\ell}$ into a $(k,m-k)$-monomial multibasis
$\mathbb{T}_{\ell}$ of $W(k,m-k)$. Next, we set
$\mathbb{S}(k,m-k):=\cup_{\ell=1}^{r} \mathbb{T}_{\ell}$. Then conditions (2--3) and the assumption that
$\mathbb{S}(k-1,m-k)$ is balanced imply that
$\mathbb{S}(k,m-k)$ is a balanced $(k,m-k)$-monomial multibasis of $W(k,m-k)$.

If $k=0$, then an analogous argument, with $\{x_i\}_{i=0}^{2g-2}$ replaced by $\{y_j\}_{j=0}^{g-3}$,
constructs $\mathbb{S}(0,m)$ from $\mathbb{S}(0,m-1)$.
\end{proof}

Next, we record an application of the preceding combinatorial lemmas, which will be used in the proof of semistability of the $2^{nd}$ Hilbert point of the Wiman curve.
\begin{example}\label{E:example} 
Let $g\geq 3$.
Consider the $(2g+2)$-dimensional linear space
\[
V:=\lspan \left\{ z^i: 0\leq i \leq 4g-4\right\} 
\big/ \lspan \left\{ z^i+z^{2g+1+i} : 0\leq i\leq 2g-6\right\}.
\]
We construct an $X$-balanced $(2,0)$-monomial multibasis of $V$ in variables $\{x_i=z^i\}_{i=0}^{2g-2}$
as follows: 
Let $\mathbb{H}_{g-3}^{2g-3}$ be the balanced $(2,0)$-monomial multibasis 
of $\lspan\{ x^i: \ g-3\leq i\leq 3g-3\}$ in variables $\{x_i\}_{i=0}^{2g-3}$, which exists by Proposition \ref{P:balanced-rational-curve}.
Set $\mathbb{T}_1 := \mathbb{H}_{g-3}^{2g-3}+\{x_{2g-2}^2\}$. 
Then $\mathbb{T}_1$ is a multibasis of $V$ of weight
\[
w_{\rho}(\mathbb{T}_1)=\frac{2(2g+1)}{2g-2}\sum_{i=0}^{2g-3}\lambda_i+2\lambda_{2g-2}.
\] 
Let $\mathbb{H}_{g-2}^{2g-2}$ be the balanced $(2,0)$-monomial multibasis 
of $\lspan \{ x^i: \ g-2\leq i \leq 3g-2\}$ in variables $\{x_i\}_{i=0}^{2g-2}$, which exists by Proposition \ref{P:balanced-rational-curve}.
Set $\mathbb{T}_2 := \mathbb{H}_{g-2}^{2g-2}+\{x_{2g-2}^2\}$. Then $\mathbb{T}_2$ is a multibasis of $V$ of weight
\[
w_{\rho}(\mathbb{T}_2)=\frac{2(2g+1)}{2g-1}\sum_{i=0}^{2g-2}\lambda_i+2\lambda_{2g-2}=\frac{2(2g+1)}{2g-1}\sum_{i=0}^{2g-3}\lambda_i+\frac{8g}{2g-1}\lambda_{2g-2}.
\] 

Evidently, a suitable combination of $\mathbb{T}_1$ and $\mathbb{T}_2$ gives an $X$-balanced multibasis of $V$ 
of weight $\frac{2(2g+2)}{2g-1}\sum_{i=0}^{2g-2} \lambda_i$. 
\end{example}

\subsection{Monomial multibases and stability} 
\label{S:monomial-bases-wiman}
The monomial (multi)bases of $\HH^0\bigl(C,\O_C(m)\bigr)$ that we use will be of the following two types.
\begin{enumerate}
\item A {\em Type I basis} consists of: 
\begin{itemize}
\item a $(m,0)$-monomial basis of $W(m,0)$; 
that is, of $(2g-2)m+1$ linearly independent degree $m$ monomials in the variables $x_i$'s.
\item a $(m-1,1)$-monomial basis of $W(m-1,1)$; 
that is, of $(2g-2)(m-1)+g-2$ linearly independent monomials that are products of a degree $m-1$ 
monomial in the variables $x_i$'s and a $y_j$ term.
\end{itemize} 
That a set of such monomials is a basis of $\HH^0\bigl(C,\O_C(m) \bigr)$ follows from 
Equation \eqref{E:isomorphism}.

 A {\em Type I multibasis} is a multibasis whose every element is a Type I basis.

\item  A {\em Type II basis} consists of:
\begin{itemize}
\item
a $(0,m)$-monomial basis of $W(0,m)$,
\item
a $(1,m-1)$-monomial basis of $W(1,m-1)$, 
\item
For $2\leq k \leq m$, a $(k,m-k)$-monomial basis of
$W(k,m-k)/r\bigl(W(k-2,m-k+2)\bigr)$.
\end{itemize} 
That a set of such monomials is a basis follows from Equation \eqref{E:isomorphism-2}.

 A {\em Type II multibasis} is a multibasis whose every element is a Type II basis.
\end{enumerate}

We pause for a moment to explain these definitions in the case of $m=2$.
\begin{enumerate}
\item A {\em Type I basis} of $\HH^0\bigl(C,\O_C(2)\bigr)$ consists 
of $4g-3$ quadratic $(2,0)$-monomials spanning $W(2,0)=\lspan\{1, \dots, z^{4g-4}\}\subset \HH^0\bigl(C,\O_C(2)\bigr)$ and of
$3g-4$ quadratic $(1,1)$-monomials spanning $W(1,1)=\lspan\{w, zw, \dots, z^{3g-5}w\}\subset \HH^0\bigl(C,\O_C(2)\bigr)$.

\item A {\em Type II basis} of $\HH^0\bigl(C,\O_C(2)\bigr)$ consists 
of $2g-5$ quadratic $(0,2)$-monomials spanning $W(0,2)=\lspan\{w^2, \dots, z^{2g-6}w^2\}\subset \HH^0\bigl(C,\O_C(2)\bigr)$; of
$3g-4$ quadratic $(1,1)$-monomials  
spanning $W(1,1)$; and of $2g+2$ quadratic $(2,0)$-monomials
that are linearly independent modulo $r\bigl(W(0,2)\bigr)$, that is, $2g+2$ monomials 
with exactly one from each pair
\[ 
(z^d, z^{d+2g+1}), \quad 0\leq d \leq 2g-6,
\]
and with the remaining $7$ being $z^{2g-5}, z^{2g-4}, z^{2g-3}, z^{2g-2}, z^{2g-1}, z^{2g}$,  and $z^{4g-4}.$
\end{enumerate}

Before proceeding with our construction of monomial bases of both types for every $m$, we illustrate our 
approach by considering the case of $m=2$, thus establishing semistability 
of the $2^{nd}$ Hilbert point of the bicanonically embedded Wiman 
curve for every $g\geq 4$.
\begin{prop}
There exist balanced Type I and Type II monomial multibases $\mathbb{B}_1$ and $\mathbb{B}_2$ of $\HH^0\bigl(C, \O_C(2)\bigr)$. 
Their weights are, respectively, 
\begin{align*}
w_{\rho}(\mathbb{B}_1)
=\frac{11g-10}{2g-1}\Lambda+\frac{3g-4}{g-2}N, \qquad \text{and} \qquad
w_{\rho}(\mathbb{B}_2) 
=\frac{7g}{2g-1}\Lambda+7N.
\end{align*}
\end{prop}
\begin{proof}
The existence of a balanced Type I multibasis follows from Proposition \ref{P:V-bases}. The existence of a balanced Type II multibasis follows from 
Proposition \ref{P:V-bases} and Example \ref{E:example}.
\end{proof}
\begin{corollary}\label{C:wiman-2nd}
The $2^{nd}$ Hilbert point of $C$ is semistable for $g\geq 4$.
\end{corollary}
\begin{proof}
For $g\geq 4$, we have $\frac{11g-10}{2g-1}>\frac{3g-4}{g-2}$ and $\frac{7g}{2g-1}<7$.
Since $\Lambda+N=0$, some positive linear combination of $w_{\rho}(\mathbb{B}_1)$ and $w_{\rho}(\mathbb{B}_2)$ is $0$
for every $\rho$ acting diagonally on the distinguished basis $\{x_0,\dots, x_{2g-2}, y_0, \dots, y_{g-3}\}$ of $\HH^0\bigl(C,\O_C(1)\bigr)$.
Semistability now follows from Lemma \ref{L:covering}.
\end{proof} 

\subsubsection{Construction of a balanced Type I multibasis}
A Type I basis is obtained by concatenating 
a $(m,0)$-monomial multibasis of $W(m,0)$ and a $(m-1,1)$-monomial multibasis of $W(m-1,1)$.
By Proposition \ref{P:V-bases}, there exists a balanced $(m,0)$-monomial multibasis of $W(m,0)$, whose weight is
\begin{align*}
m\frac{(2m(g-1)+1)}{(2g-1)}\Lambda,
\end{align*}
and a balanced $(m-1,1)$-monomial multibasis of $W(m-1,1)$, whose weight is 
\begin{align*}
(m-1)\frac{(2g-2)(m-1)+g-2)}{(2g-1)}\Lambda+\frac{\bigl((2g-2)(m-1)+(g-2)\bigr)}{(g-2)}N.
\end{align*}

Summarizing, we obtain the following result.
\begin{prop}\label{P:type-I}
There is a Type I multibasis of $\HH^0\bigl(C,\O_C(m)\bigr)$ of weight
\begin{equation}\label{E:weight-type-I}
 \frac{\bigl((4g-4)m^2-(3g-3)m+g\bigr)}{(2g-1)}\Lambda+\frac{\bigl((2g-2)m-g\bigr)}{(g-2)}N.
\end{equation}
\end{prop}
\begin{remark}\label{remark}
We note that in Equation \eqref{E:weight-type-I}, the coefficient of $\Lambda$ is greater than the coefficient of $N$ 
for all values of $g\geq 3$ and all values of $m\geq 2$, with the sole exception of $(g,m)=(3,2)$ for which
we get $\frac{23}{5}\Lambda+5N$. It is easy to see that in this exceptional case, the $2^{nd}$ Hilbert 
point of the bicanonically embedded genus $3$ Wiman curve is, in fact, {\em non-semistable}.
\end{remark}
\subsubsection{Construction of a Type II basis}
In this section we construct a (balanced) Type II multibasis of $\HH^0\bigl(C,\O_C(m)\bigr)$. We begin with a preliminary result.
\begin{lemma}\label{L:type-II-key} {\em (a)} Suppose $k\geq 3$. Then for $0\leq i\leq 2g-2$ and 
$0\leq \epsilon \ll 1$, there is a $(k,m-k)$-monomial multibasis of $W(k,m-k)/r\bigl(W(k-2,m-k+2)\bigr)$ whose
weight is 
\begin{equation}\label{E:type-II-piece-3}
\left(k\frac{(2g+2)}{(2g-1)}-\frac{\epsilon}{2g-1}\right)\Lambda+(m-k)\frac{(2g+2)}{(g-2)}N+\epsilon \lambda_i.
\end{equation}
{\em (b)} Suppose $m-k\geq 1$. Then for $0\leq j\leq g-3$ and 
$0\leq \delta \ll 1$, there is a $(k,m-k)$-monomial multibasis of $W(k,m-k)/r\bigl(W(k-2,m-k+2)\bigr)$ whose
weight is 
\begin{equation}\label{E:type-II-piece-4}
k\frac{(2g+2)}{(2g-1)}\Lambda+\left((m-k)\frac{(2g+2)}{(g-2)}-\frac{\delta}{g-2}\right)N+\delta \nu_j.
\end{equation}

\end{lemma}

\begin{proof}
We identify $W(k,m-k)/r\bigl(W(k-2,m-k+2)\bigr)$ with the vector space  
\[
\lspan \{z^d w^{m-k} \ : \ 0\leq d \leq (2g-2)k+(g-3)(m-k)\}
\] 
modulo the relations
\[
z^{d+2g+1}w^{m-k}+z^{d}w^{m-k} =0, \quad 0\leq d\leq (2g-2)k+(g-3)(m-k)-(2g+2).
\]

We define a set of $(k,m-k)$-monomials that form a basis $W(k,m-k)/r\bigl(W(k-2,m-k+2)\bigr)$, and which
depends on three parameters: $i\in \{0,\dots, 2g-2\}$,  $j\in \{0,\dots, g-3\}$, and $u\in \{0,1\}$:
\begin{equation}\label{E:S}
\mathbb{S}_u(i,j):=\{x_{2g-2}^k y_{g-3}^{m-k}\}\cup \left(x_i^{k-2}y_{j}^{m-k} \times \mathbb{T}_u\right),
\end{equation}
where
\begin{itemize}
\item
 $\mathbb{T}_0:=\mathbb{H}^{2g-3}_{g-3}$ is the {\em quadratic} monomial multibasis of  $\lspan\{ z^i: \ g-3\leq i \leq 3g-3\}$ in variables
 $\{x_0,\dots, x_{2g-3}\}$, which exists by
Proposition \ref{P:balanced-rational-curve} and has weight 
\[
w_{\rho}(\mathbb{T}_0)=\frac{2(2g+1)}{2g-2}(\Lambda-\lambda_{2g-2}).
\]
\item $\mathbb{T}_1:=\mathbb{H}^{2g-2}_{g-2}$ is the {\em quadratic} monomial multibasis of $\lspan\{z^i: \ g-2 \leq i \leq 3g-2\}$ 
in variables $\{x_0, \dots, x_{2g-2}\}$, which exists by
Proposition \ref{P:balanced-rational-curve} and has weight 
\[
w_{\rho}(\mathbb{T}_1)=\frac{2(2g+1)}{2g-1}\Lambda.
\]
\end{itemize}
Setting $\nu(k,j):=(2g+1)(m-k)\nu_j+(m-k)\nu_{g-3}$,
we deduce that 
\begin{align}
w_{\rho}\bigl(\mathbb{S}_0(i,j)\bigr) &=\frac{2(2g+1)}{2g-2}(\Lambda-\lambda_{2g-2})+(2g+1)(k-2)\lambda_{i}+k\lambda_{2g-2}+\nu(k,j),\\ 
w_{\rho}\bigl(\mathbb{S}_1(i,j)\bigr) &=\frac{2(2g+1)}{2g-1}\Lambda+(2g+1)(k-2)\lambda_{i}+k\lambda_{2g-2}+\nu(k,j).
\end{align}
Since $\sum_{i=0}^{2g-3}\lambda_i=\Lambda-\lambda_{2g-2}$, the mutibasis 
$\mathbb{S}_0:=\cup_{i=0}^{2g-3} \mathbb{S}_0(i,j)$ has weight
\begin{equation*}
k\frac{(2g+1)}{2g-2}(\Lambda-\lambda_{2g-2})+k\lambda_{2g-2} +\nu(k,j).
\end{equation*} 
If $a+b+c=1$, then $\mathbb{S}_1:=a\cdot \mathbb{S}_0 \cup b\cdot \mathbb{S}_1(2g-2,j)\cup c\cdot \mathbb{S}_1(i,j)$ has weight 
\begin{multline}\label{E:eq}
\left[ ak\frac{2g+1}{2g-2}+2(b+c)\frac{2g+1}{2g-1}\right]\Lambda 
+\left[ ck+b\bigl((2g+2)k-2(2g+1)\bigr)-\frac{3ak}{2g-2}\right] \lambda_{2g-2}\\
+\left[ c(2g+1)(k-2)\right]\lambda_{i}+\nu(k,j).
\end{multline}
For any small non-negative $c$, we can find $a$ and $b$ in $[0,1]$ satisfying $a+b+c=1$ and such that 
the coefficient of $\lambda_{2g-2}$ in \eqref{E:eq} equals $0$. If we additionally require that $c=0$, which then 
also determines $a$ and $b$, the $\Lambda$ coefficient in \eqref{E:eq} simplifies to 
$k\dfrac{(2g+2)}{(2g-1)}$. For $c=\epsilon$, it follows that $\mathbb{S}_1$ has weight 
\[
\left(k\dfrac{(2g+2)}{(2g-1)}-\frac{\epsilon}{2g-1}\right)\Lambda+\epsilon \lambda_i+\nu(k,j).
\]
Recall that $\nu(k,j)=(2g+1)(m-k)\nu_j+(m-k)\nu_{g-3}$. Since $(2g+1)(m-k)>(m-k)$, 
an averaging argument with $\nu$'s, analogous to the one given above for $\lambda$'s, shows that there exist multibases of weights
given by Equations \eqref{E:type-II-piece-3} and \eqref{E:type-II-piece-4}.
\end{proof}

\begin{prop}\label{P:type-II} Let $m\geq 3$. For $0\leq i \leq 2g-2$, $0\leq j \leq g-3$, and $0\leq \epsilon, \delta \ll 1$,
there exists a Type II multibasis of $\HH^0\bigl(C,\O_C(m)\bigr)$ of weight $a \Lambda + b N +\epsilon \lambda_i+\delta \nu_j$,
where
\begin{align*}
a&=\dfrac{1}{2g-1} \left((g+1)m^2+(2g-2)m-g)\right)-\frac{\epsilon}{2g-1}, \\
b&=\dfrac{1}{g-2} \left((3g-5)m^2-(3g-3)m+g \right)-\frac{\delta}{g-2};
\end{align*}
in particular $a<b$. 
\end{prop}
\begin{proof}
We begin with a balanced $(0,m)$-monomial multibasis of $W(0,m)$
which exists by Proposition \ref{P:V-bases} and whose weight is
\begin{equation}\label{E:type-II-piece-1}
\dfrac{m\bigl((g-3)m+1\bigr)}{(g-2)}N.
\end{equation}
Next, we take
a balanced $(1,m-1)$-monomial multibasis of $W(1,m-1)$, which again exists by Proposition \ref{P:V-bases}. 
Its weight is
\begin{equation}\label{E:type-II-piece-2}
(m-1)\dfrac{\bigl((g-3)(m-1)+2g-1\bigr)}{(g-2)}N+\dfrac{\bigl((g-3)(m-1)+2g-1\bigr)}{(2g-1)}\Lambda.
\end{equation}

By Lemma \ref{L:type-II-key} there exists a multibasis of $W(k,m-k)/r\bigl(W(k-2,m-k+2)\bigr)$
of weight
\[
\omega_k:=k\frac{(2g+2)}{(2g-1)}\Lambda + (m-k)\frac{(2g+2)}{(g-2)}N,
\]
for $2\leq k\leq m$.
Moreover, by the same lemma, there exists a $(3,m-3)$-monomial multibasis of $W(3,m-3$) of 
weight 
\[
\omega'_3:=\left(3\frac{(2g+2)}{(2g-1)}-\frac{\epsilon}{(2g-1)}\right)\Lambda+(m-3)\frac{(2g+2)}{(g-2)}N+\epsilon \lambda_i,
\]
for any small non-negative $\epsilon$ and any $i$, and
there also exists a multibasis of $W(2,m-2)$ of 
weight 
\[
\omega'_2:=2\frac{(2g+2)}{(2g-1)}\Lambda+\left((m-2)\frac{(2g+2)}{(g-2)}-\frac{\delta}{(g-2)}\right)N+\delta \nu_j,
\]
again for any small non-negative $\delta$ and any $j$. 
Concatenating the above bases, we obtain a Type II multibasis.
If we set $\epsilon=\delta=0$, the weight of the resulting multibasis is
\begin{multline*}
\dfrac{m\bigl((g-3)m+1\bigr)}{(g-2)}N+(m-1)\dfrac{\bigl((g-3)(m-1)+2g-1\bigr)}{(g-2)}N+\dfrac{\bigl((g-3)(m-1)+2g-1\bigr)}{(2g-1)} \Lambda    \\ 
+\sum_{k=2}^{m} \omega_k = \dfrac{1}{2g-1} \left((g+1)m^2+(2g-2)m-g)\right)\Lambda+\dfrac{1}{g-2} \left( (3g-5)m^2-(3g-3)m+g \right)N.
\end{multline*}
The result follows.
\end{proof}

We are now ready to prove Theorem \ref{T:wiman} 
in the case of $m\geq 3$. 

\begin{corollary}[Stability of Wiman curves]\label{T:wiman-m-3}
The $m^{th}$ Hilbert point of the bicanonically embedded Wiman curve $C$ of genus $g\geq 3$ is stable for every $m\geq 3$.
\end{corollary}
\begin{proof} 
Lemma \ref{L:wiman-mult} and Proposition \ref{P:kempf} reduce the verification of stability of $C$ to verifying stability 
with respect to one-parameter subgroups acting diagonally 
on the distinguished basis $\{x_0, \dots, x_{2g-2}, y_0, \dots, y_{g-3}\}$ of $\HH^0\bigl(C,\O_C(1) \bigr)$. 
To prove stability with respect to such one-parameter subgroup $\rho$, we need to find a monomial basis of $\HH^0\bigl(C,\O_C(m)\bigr)$ 
of negative $\rho$-weight. 
By taking a suitable linear combination of the Type I monomial multibasis of Proposition \ref{P:type-I}
and the Type II monomial multibasis of Proposition \ref{P:type-II}, we can now construct a monomial multibasis of weight 
$\epsilon \lambda_i+\delta \nu_{j}$,
where $0 \leq \epsilon, \delta \ll 1$ are arbitrary and indices $i, j$ can be chosen freely. Since at least one of the weights
$\{\lambda_0, \dots, \lambda_{2g-2}, \nu_0, \dots, \nu_{g-3}\}$ is negative, the claim follows.
\end{proof}

\bibliography{stability-bib}{}
\bibliographystyle{alpha}

\end{document}